\documentclass[11pt,a4paper,reqno]{amsart}
\pdfoutput=1

\usepackage[T1]{fontenc}
\usepackage[utf8]{inputenc}
\usepackage{lmodern}
\usepackage[english]{babel}
\usepackage{csquotes}
\usepackage[activate]{pdfcprot}

\usepackage{datetime}

\usepackage{amssymb}
\usepackage{amscd}
\usepackage{amsfonts}
\usepackage{mathtools}

\usepackage{dsfont}

\usepackage{pgfplots}
\pgfplotsset{compat=1.16}
\usetikzlibrary{arrows}
\usetikzlibrary[patterns]
\definecolor{wqwqwq}{rgb}{0.3764705882352941,0.3764705882352941,0.3764705882352941}
\definecolor{xdxdff}{rgb}{0.6588235294117647,0.6588235294117647,0.6588235294117647}

\usepackage[pdftex,colorlinks,linkcolor=black,citecolor=black,filecolor=black]{hyperref}

\renewcommand\labelenumi{(\roman{enumi})}
\renewcommand\theenumi\labelenumi

\numberwithin{equation}{section}
\numberwithin{figure}{section}
\theoremstyle{plain}
\newtheorem{theorem}{Theorem}[section]
\newtheorem{proposition}[theorem]{Proposition}
\newtheorem{lemma}[theorem]{Lemma}
\newtheorem{corollary}[theorem]{Corollary}
\theoremstyle{remark}
\newtheorem{remarkx}[theorem]{Remark}
\newenvironment{remark}
{\pushQED{\qed}\remarkx}
 {\popQED\endremarkx}
\theoremstyle{definition}
\newtheorem{definition}[theorem]{Definition}
\newtheorem{examplex}[theorem]{Example}
\newenvironment{example}
{\pushQED{\qed}\examplex}
 {\popQED\endexamplex}

\newcommand\G{\mathcal{G}}
\newcommand\calS{\mathcal{S}}
\newcommand\B{\mathcal{B}}
\newcommand\g{\mathfrak{g}}
\DeclareMathOperator{\Aut}{Aut}
\DeclareMathOperator{\supp}{supp}
\DeclareMathOperator{\Ad}{Ad}
\newcommand{\Pro}{\mathbb{P}}
\DeclareMathOperator{\pr}{pr}
\newcommand{\E}{\mathbb{E}}
\DeclareMathOperator{\SL}{SL}

\newcommand\C{\mathbb{C}}
\newcommand\N{\mathbb{N}}
\newcommand\Q{\mathbb{Q}}
\newcommand\R{\mathbb{R}}
\newcommand\Z{\mathbb{Z}}

\newcommand{\dd}{\mathop{}\!\mathrm{d}}

\DeclarePairedDelimiter\br{(}{)}
\DeclarePairedDelimiter\sqbr{[}{]}
\DeclarePairedDelimiter\abs{\lvert}{\rvert}
\DeclarePairedDelimiter\norm{\lVert}{\rVert}

\providecommand\for{}
\newcommand\SetSymbol[1][]{%
\nonscript\:#1\vert
\allowbreak
\nonscript\:
\mathopen{}}
\DeclarePairedDelimiterX\set[1]{\lbrace}{\rbrace}{%
\renewcommand\for{\SetSymbol[\delimsize]}
#1
}

\newcommand{\euler}{\mathrm{e}}

\renewcommand{\epsilon}{\varepsilon}

\title[Spread out random walks]{Spread out random walks on homogeneous spaces}

\author{Roland Prohaska}
\address{Departement Mathematik, ETH Z\"{u}rich, R\"{a}mistrasse 101, 8092 Z\"{u}rich, Switzerland}
\email{roland.prohaska@math.ethz.ch}

\date{\usdate\today}

\begin{document}

\begin{abstract}
A measure on a locally compact group is called spread out if one of its convolution powers is not singular with respect to Haar measure. 
Using Markov chain theory, we conduct a detailed analysis of random walks on homogeneous spaces with spread out increment distribution. 
For finite volume spaces, we arrive at a complete picture of the asymptotics of the $n$-step distributions: 
They equidistribute towards Haar measure, often exponentially fast and locally uniformly in the starting position. 
In addition, many classical limit theorems are shown to hold. 
In the infinite volume case, we prove recurrence and a ratio limit theorem for symmetric spread out random walks on homogeneous spaces of at most quadratic growth. 
This settles one direction in a long-standing conjecture. 
\end{abstract}
\maketitle
\section{Introduction}\label{sec:intro}
Let $G$ be a $\sigma$-compact locally compact metrizable group, $\Gamma\subset G$ a discrete subgroup, and $X$ the homogeneous space $G/\Gamma$. 
A Borel probability measure $\mu$ on $G$ defines a random walk on $X$: 
A step corresponds to choosing a group element $g\in G$ according to $\mu$ and then moving from the current location $X\ni x$ to $gx$. 
Given a starting point $x_0\in X$, we may represent the location $\Phi_n$ after $n$ steps as 
\begin{align}\label{Phi_def}
\Phi_n=Y_n\dotsm Y_1x_0,
\end{align}
where $(Y_k)_{k\in\N}$ is a sequence of i.i.d.\ random variables in $G$ with common law $\mu$. 

Over the past few decades, a substantial amount of research has been dedicated to understanding the long-term behavior of such random walks for non-abelian groups $G$, especially semisimple real Lie groups, with cornerstone developments e.g.\ by Furstenberg~\cite{F}, Eskin--Margulis~\cite{EM}, and Benoist--Quint~\cite{BQ1,BQ_rec,BQ2,BQ3}. 
Ideally, one would like to quantitatively describe the asymptotics of the random walk in terms of some natural, \enquote{stable} limiting distribution on $X$. 
A prominent candidate for the latter is a \emph{Haar measure} $m_X$ on $X$, by which we shall mean a non-trivial $G$-invariant Radon measure on $X$ (if one exists). 
In case $X$ admits a finite Haar measure, we assume that $m_X$ is normalized to be a probability measure, call $\Gamma$ a \emph{lattice}, and say that $X$ has finite volume. 
Otherwise, we say that $X$ has infinite volume. 
According to this distinction, the discussion splits into two cases. 
\subsection{Finite Volume Spaces}\label{subsec:fin_vol_intro}
To illustrate the kind of description of the asymptotic behavior of a random walk we are interested in, let us reproduce one of the main results of~\cite{BQ3} as motivating example. 
For the statement, recall that a probability measure $\nu$ on $X$ is called \emph{homogeneous} if there exists a closed subgroup $H$ of $G$ and a point $x\in X$ such that $\supp(\nu)=Hx$ and $\nu$ is $H$-invariant. 
\begin{theorem}[{Benoist--Quint~\cite{BQ3}}]\label{thm:BQ}
Let $G$ be a real Lie group, $\Gamma\subset G$ a lattice, and $\mu$ a compactly supported probability measure on $G$. 
Suppose that the closed subsemigroup $\calS$ generated by $\supp(\mu)$ has the property that the Zariski closure of $\Ad(\calS)$ in $\Aut(\g)$ is Zariski connected, semisimple, and has no compact factors. 
Then for every $x_0\in X$ there is a homogeneous probability measure $\nu_{x_0}$ with $\supp(\nu_{x_0})=\overline{\calS x_0}$ and such that 
\begin{align*}
\frac1n\sum_{k=0}^{n-1}\mu^{*k}*\delta_{x_0}\longrightarrow \nu_{x_0}
\end{align*}
as $n\to\infty$ in the weak* topology. 
In particular, whenever $\calS x_0$ is dense in $X$ the limiting distribution $\nu_{x_0}$ is the normalized Haar measure $m_X$. 
\end{theorem}
Some questions left open by this theorem are listed by Benoist--Quint at the end of their survey~\cite{BQ:intro}. 
Among other things, they ask the following: 
\begin{enumerate}
\item[(Q1)] Does convergence also hold for the non-averaged laws $\mathcal{L}(\Phi_n)=\mu^{*n}*\delta_{x_0}$? 
\item[(Q2)] Can the convergence be made effective? 
\end{enumerate}

Answers to these questions are known only in special cases: 
Breuillard~\cite{Br} established (Q1) for certain measures supported on unipotent subgroups, and recently Buenger~\cite{Bue} was able to positively answer (Q1) and (Q2) for some sparse solvable measures. 
In this article, we add to this list the class of aperiodic spread out measures. 
\begin{definition}\label{def:spread_out_aper}
Let $\mu$ be a probability measure on $G$. 
\begin{itemize}
\item The measure $\mu$ is called \emph{spread out} if for some $n_0\in\N$ the convolution power $\mu^{*n_0}$ is not singular with respect to Haar measure on $G$. 
\item Let $\G$ denote the closed subgroup of $G$ generated by $\supp(\mu)$. 
Then we call $\mu$ \emph{aperiodic} if $\mu$ is not supported on a coset of a proper normal open subgroup of $\G$ containing the commutator subgroup $\sqbr{\G,\G}$. 
\end{itemize}
\end{definition}
As we shall see, the qualitative behavior of spread out random walks on finite volume homogeneous spaces can be understood in great detail, and in fact for a much larger class of groups than (semisimple) real Lie groups. 
In particular, no connectedness assumption needs to be imposed, so that e.g.\ discrete or $p$-adic groups are naturally included in our setup. 
\begin{theorem}\label{thm:main_thm}
Let $\Gamma\subset G$ be a lattice and $\mu$ an aperiodic spread out probability measure on $G$. 
Then for every $x_0\in X$ the orbit $\G x_0$ is clopen in $X$ and we have 
\begin{align}\label{equidist}
\norm{\mu^{*n}*\delta_{x_0}-m_{\G x_0}}\longrightarrow 0
\end{align}
as $n\to\infty$, where $m_{\G x_0}$ denotes the normalized Haar measure on $\G x_0$ and $\norm{\cdot}$ is the total variation norm. 
If the random walk additionally admits a continuous and everywhere finite Lyapunov function \textup{(}see~\textup{\S\ref{subsec:eff_mix})}, then there is a constant $\kappa>0$ such that for every compact subset $K\subset X$ and $n\in\N$ we have 
\begin{align*}
\sup_{x\in K}\norm*{\mu^{*n}*\delta_x-m_{\G x}}\ll_K \mathrm{e}^{-\kappa n}.
\end{align*} 
For example, this holds when $G$ is a connected semisimple real algebraic group without compact factors and $\mu$ has compact support. 
\end{theorem}
For a statement without the aperiodicity assumption we refer the reader to the discussion in~\S\ref{sec:consequences}. 

In two special cases, the above result takes a particularly simple form. 
One of them is when $X$ is connected, the other when $\mu$ is adapted. 
\begin{definition}\label{def:adapted}
A probability measure $\mu$ on $G$ is called \emph{adapted} if the closed subgroup $\G$ generated by $\supp(\mu)$ coincides with $G$. 
\end{definition}
\begin{corollary}\label{cor:conn}
Let $\Gamma\subset G$ be a lattice and $\mu$ a spread out probability measure on $G$. 
Suppose that $X$ is connected or that $\mu$ is additionally adapted and aperiodic. 
Then for every $x_0\in X$ we have 
\begin{align*}
\norm{\mu^{*n}*\delta_{x_0}-m_X}\longrightarrow 0
\end{align*}
as $n\to\infty$, where $m_X$ denotes the normalized Haar measure on $X$. 
\end{corollary}
\begin{remark}
In the literature on spread out random walks it has been customary to restrict attention to adapted measures $\mu$ (\cite{GScho,HR,Scho:p-adic,Scho:RW,Scho:rec_RW}). 
This is indeed often justified, since one can replace $G$ by $\G=\overline{\langle \supp(\mu)\rangle}$ (see Lemma~\ref{lem:adapted}). 
However, as a consequence one must also replace $X$ by an orbit $\G x$, which is not always desirable. 
Hence, we emphasize that in the case of a connected space $X$, adaptedness (or aperiodicity) of $\mu$ are not needed as assumptions in the above corollary, distinguishing this result from the existing literature. 
\end{remark}
Our approach is to analyze the random walk given by a spread out measure $\mu$ from the viewpoint of general state space Markov chain theory. 
The key observation is that it is a positive Harris recurrent $T$-chain on every $\G$-orbit in $X$. 
A connectedness assumption can then be used to establish transitivity (i.e.\ $\G x=X$) and rule out periodic behavior. 
Feeding all of this into the general theory, we obtain our results. 

As a matter of fact, exploring the extent to which Markov chain theory can be of use in the study of random walks on finite volume homogeneous spaces has been one of the motivations for the present work. 
As they note, already Benoist--Quint's approach was inspired by Markov chain methods (\cite[p.~702]{BQ2}); however, they could not directly apply available results, since the key assumption of $\psi$-irreducibility was not satisfied in the applications they had in mind (\cite[p.~703]{BQ2}). 
A natural question is when this assumption is satisfied. 
As part of our discussion, we show that this is the case precisely for spread out measures (see Proposition~\ref{prop:spread_out_irr} and Corollary~\ref{cor:fin_vol_irred}). 
\subsection{Infinite Volume Spaces}\label{subsec:inf_vol_intro}
Most of the qualitative analysis underlying Theorem~\ref{thm:main_thm} can also be carried out in the infinite volume case. 
For the upgrade to quantitative information though, one has to deal with an additional issue: recurrence of the random walk. 
The following dichotomy theorem of Hennion--Roynette describes the situations that can occur for spread out random walks. 
We write $\Pro_x$ for a probability measure under which the random walk \eqref{Phi_def} starts at $x\in X$ and $\E_x$ for the associated expectation (see~\S\ref{subsec:prelim}). 
\begin{theorem}[Hennion--Roynette~\cite{HR}]\label{thm:HR}
Let $\mu$ be an adapted spread out probability measure on $G$. 
Suppose that $X$ admits a Haar measure $m_X$. 
Then either 
\begin{enumerate}
\item all states $x\in X$ are \emph{topologically Harris recurrent}, meaning that 
\begin{align*}
\Pro_x\sqbr{\Phi_n\in B\text{ infinitely often}}=1
\end{align*}
for all neighborhoods $B$ of $x$, or 
\item all states $x\in X$ are \emph{topologically transient}, meaning that for some neighborhood $B$ of $x$ 
\begin{align*}
\E_x\sqbr*{\sum_{n=1}^\infty\mathds{1}_{\Phi_n\in B}}<\infty.
\end{align*}
\end{enumerate}
Accordingly, the random walk on $X$ given by $\mu$ is called \emph{topologically Harris recurrent} or \emph{topologically transient}. 
\end{theorem}
It is not difficult to see that spread out random walks on finite volume spaces are topologically Harris recurrent. 
Indeed, Kakutani's random ergodic theorem (\cite{Kak}, see also~\cite{FS}) implies that $m_X$-a.e.\ point satisfies the condition in (i). 
In general, what spaces $X$ admit topologically Harris recurrent spread out random walks is a difficult question, extensively studied by Schott~\cite{GScho,Scho:p-adic,Scho:non_amenable,Scho:RW,Scho:rec_RW}, which turns out to be intimately linked to the growth of the space. 
\begin{definition}\label{def:growth}
Suppose that $X$ admits a Haar measure $m_X$. 
Then $X$ is said to have \emph{polynomial growth of degree at most $d$} if there exists a generating relatively compact neighborhood $V$ of the identity in $G$ and $x\in X$ such that 
\begin{align*}
\limsup_{n\to\infty}\frac{m_X(V^nx)}{n^d}<\infty.
\end{align*}
In this case, it can be shown that the above holds for all choices of $x$ and $V$ (\cite{GScho}). 
When $d=2$ we say the growth is \emph{at most quadratic}. 
\end{definition}

Analogous to the more classical case of random walks on groups (for which see e.g.~\cite{GR} and the references therein), the \emph{quadratic growth conjecture} states that the homogeneous space $X=G/\Gamma$ admits topologically Harris recurrent spread out random walks if and only if it is of at most quadratic growth. 
For example, this is known to hold if $G$ is a connected real Lie group of polynomial growth (Hebisch--Saloff-Coste \cite[\S10]{HS_gaussian}) or a $p$-adic algebraic group of polynomial growth (Raja--Schott~\cite{Scho:p-adic}). 
In this paper, we show that one implication holds in general. 
\begin{theorem}\label{thm:quadratic_recurrence}
Suppose that $X$ admits a Haar measure and has at most quadratic growth. 
Let $\mu$ be an adapted symmetric spread out probability measure on $G$ with compact support. 
Then the random walk on $X$ given by $\mu$ is topologically Harris recurrent. 
\end{theorem}
Here the requirement of $\mu$ being \emph{symmetric} means that $\mu(B)=\mu(B^{-1})$ for all measurable $B\subset G$. 

Once Harris recurrence is established, we have an analogue of \eqref{equidist} in the form of a ratio limit theorem. 
\begin{theorem}\label{thm:ratio_limit_thm}
Let $\mu$ be an adapted spread out probability measure on $G$. 
Suppose that $X$ admits a Haar measure $m_X$ and that the random walk on $X$ given by $\mu$ is topologically Harris recurrent. 
Then for any $x_1,x_2\in X$ and two bounded measurable functions $f_1,f_2$ on $X$ with compact support such that $f_2\ge 0$ and $\int_Xf_2\dd m_X\neq 0$ we have 
\begin{align*}
\frac{\sum_{j=0}^n\int_Xf_1\dd(\mu^{*j}*\delta_{x_1})}{\sum_{j=0}^n\int_Xf_2\dd(\mu^{*j}*\delta_{x_2})}\longrightarrow\frac{\int_Xf_1\dd m_X}{\int_Xf_2\dd m_X}
\end{align*}
as $n\to\infty$. 
If $\mu$ is additionally symmetric and aperiodic, then 
\begin{align}\label{strong_ratio}
\frac{\int_Xf_1\dd(\mu^{*n}*\nu_1)}{\int_Xf_2\dd(\mu^{*n}*\nu_2)}\longrightarrow\frac{\int_Xf_1\dd m_X}{\int_Xf_2\dd m_X}
\end{align}
as $n\to\infty$ for any two probability measures $\nu_1,\nu_2\ll m_X$ with bounded density. 
\end{theorem}
\begin{remark}\label{rmk:conjecture}
We conjecture that \eqref{strong_ratio} also holds with Dirac measures $\delta_{x_1},\delta_{x_2}$ in place of $\nu_1,\nu_2$ for arbitrary $x_1,x_2\in X$. 
Unfortunately, we can only prove this under the additional condition 
\begin{align}\label{sufficient}
\limsup_{n\to\infty}\frac{(\mu^{*n}*\delta_{x_i})(A)}{(\mu^{*n}*m_A)(A)}\le 1
\end{align}
for $i=1,2$, where $A$ is a certain \enquote{small} subset of $X$ (see the proof of Theorem~\ref{thm:ratio_limit_thm} in~\S\ref{subsec:ratio_limit}) and $m_A=\frac{1}{m_X(A)}m_X|_A$ is the normalized restriction of $m_X$ to $A$. 
\end{remark}

A standard example to which the previous results apply is the following. 
\begin{example}[Covering spaces]
Let $G$ be a connected real Lie group, $\Gamma'\subset G$ a cocompact lattice and  $\Gamma\subset \Gamma'$ a normal subgroup. 
Then $X=G/\Gamma$ is a $\Gamma'/\Gamma$-cover of $G/\Gamma'$, so that $X$ has at most quadratic growth if this is the case for the discrete group $\Gamma'/\Gamma$. 
\end{example}
For simple non-compact Lie groups of real rank $1$, symmetric finitely supported measures $\mu$, and $\Gamma'/\Gamma\cong\Z$ or $\Z^2$, recurrence in the above example has been known (Conze--Guivarc'h \cite[Proposition~4.5]{CG}). 
The corresponding recurrence result under our conditions is new. 
\subsection{Examples of Spread Out Measures}
We conclude this introduction by shedding some more light on the nature of spread out measures. 
Naturally, the first examples coming to mind are measures absolutely continuous with respect to Haar measure on $G$. 
However, the class of spread out measures is much larger and also contains many interesting singular measures, as the following examples aim to illustrate. 
\begin{example}[Affine random walks on the torus]\label{ex:torus}
An affine transformation on the torus $\mathbb{T}^n=\R^n/\Z^n$ is a map of the form 
\begin{align}\label{affine_trans}
\mathbb{T}^n\ni x\mapsto gx+v,
\end{align}
where $g\in\SL_n(\Z)$ is a unimodular integer matrix and $v\in\R^n$ is a translation vector. 
They fit into our setup in the following way: 
The group $G$ is the semidirect product $\SL_n(\Z)\ltimes \R^n$ with group law $(g,v)(h,w)=(gh,gw+v)$ and the lattice is given by $\Gamma=\SL_n(\Z)\ltimes \Z^n$. 
Then $\mathbb{T}^n\cong X=G/\Gamma$, an element $(g,v)\in G$ acts on $x\in X$ precisely by \eqref{affine_trans}, and an affine random walk on the torus is described by a measure $\mu$ on $G$. 

We shall now explain when such a measure $\mu$ is spread out in two cases. 
Let us write $\lambda_v$ for the pushforward of a measure $\lambda$ on $\R$ to a line $\R v\subset \R^n$ via $t\mapsto tv$. 
\begin{enumerate}
\item The simplest case is when the linear part of the random walk is deterministic, given by a single matrix $a\in\SL_n(\Z)$. 
For the measure $\mu$, this means that $\mu=\delta_a\otimes \mu_{\mathrm{trans}}$ for some probability measure $\mu_{\mathrm{trans}}$ on $\R^n$ giving the distribution of the translational part. 
When $\mu_{\mathrm{trans}}$ has $n$-dimensional density, already $\mu$ is not singular with respect to Haar measure $m_G=m_{\mathrm{count}}\otimes m_{\R^n}$ on $G$, and so in particular spread out. 
However, we can do much better than that: 
It often suffices for $\mu_{\mathrm{trans}}$ to have density in only one direction. 
More precisely, let $\lambda$ be a probability measure on $\R$ that is not singular with respect to Lebesgue measure, $v\in\R^n$ a unit vector, and $\mu_{\mathrm{trans}}=\lambda_v$. 
Then $\mu=\delta_a\otimes \lambda_v$ is spread out if and only if $\set{v,av,\dots,a^{n-1}v}$ spans $\R^n$. 
\item A similar characterization is possible when the linear and translational parts of $\mu$ are only assumed to be independent, i.e.\ if $\mu=\mu_{\mathrm{lin}}\otimes \mu_{\mathrm{trans}}$ for some probability measures $\mu_{\mathrm{lin}}$ on $\SL_n(\Z)$ and $\mu_{\mathrm{trans}}$ on $\R^n$. 
Aiming to introduce as little density as possible, we again suppose $\mu_{\mathrm{trans}}=\lambda_v$ for some $\lambda$ non-singular with respect to Lebesgue measure on $\R$ and a unit vector $v\in\R^n$. 
Then $\mu=\mu_{\mathrm{lin}}\otimes \lambda_v$ is spread out if and only if $v$ is not contained in a proper $\supp(\mu_{\mathrm{lin}})$-invariant subspace of $\R^n$. 
For example, this is automatically the case under the common assumption that the semigroup $S$ generated by $\supp(\mu_{\mathrm{lin}})$ acts irreducibly on $\R^n$. 
\end{enumerate}

The justification of the claims in the two points above is the following observation: 
If $\eta$ is a measure on a subspace $V\subset \R^n$ non-singular with respect to Lebesgue measure on that subspace, then by definition of the group law on $G$ we have 
\begin{align*}
\mu*(\delta_s\otimes \eta)\gg \delta_{as}\otimes \eta'
\end{align*}
for any $a\in\supp(\mu_{\mathrm{lin}})$ and $s\in \SL_n(\Z)$, where $\eta'$ is supported on $V'=\R v+aV$ and again non-singular with respect to Lebesgue measure on that space. 
In other words, in each convolution step we can pass from a density on $V$ to a density on $V'=\R v+aV$ for any $a\in\supp(\mu_{\mathrm{lin}})$. 
Starting from $\eta= \lambda_v$ and $V=\R v$, the question of whether $\mu$ is spread out is thus equivalent to asking if it is possible to reach $V'=\R^n$ in finitely many such steps. 
With a little work, this yields the stated conditions. 
\end{example}
\begin{example}
Let $G=\SL_2(\R)$ and 
\begin{align*}
U=\set*{u_s=\begin{pmatrix}1&s\\0&1\end{pmatrix}\for s\in \R}\cong \R
\end{align*}
be the upper unipotent subgroup. 
Furthermore, let $f\colon U\to[0,\infty)$ be any continuous density with $f(u_0)>0$ and $\int_Uf\dd s=1$, set $\dd\mu_U=f\dd s$ and $u_-=(\begin{smallmatrix}1&0\\1&1\end{smallmatrix})$. 
Then for the probability measure 
\begin{align*}
\mu=\tfrac12(\mu_U+\delta_{u_-}),
\end{align*}
the fifth convolution power $\mu^{*5}$ has a non-trivial absolutely continuous component, as a calculation shows. 
(For example, observe that in a neighborhood of the origin, $(a,b,c)\mapsto u_au_-u_bu_-u_c$ is a smooth chart of a neighborhood of $u_-^2$ inside $G$.) 
Hence, $\mu$ is singular with respect to Haar measure, yet spread out. 
\end{example}
\subsection*{Acknowledgments}
The author would like to thank Cagri Sert for suggesting the use of Markov chain theory in the study of random walks on homogeneous spaces, as well as for numerous helpful discussions and remarks on this article. 
Thanks also go to Manfred Einsiedler for pointing out the possibility of including infinite volume spaces and Andreas Wieser for valuable comments on draft versions of the article. 
\section{Markov Chain Theory for Random Walks}\label{sec:markov_theory}
In this section, we lay the foundations for all following discussions. 
We review the relevant concepts and results from general state space Markov chain theory in~\S\ref{subsec:prelim}, and make the connection to spread out random walks in~\S\ref{subsec:T_chains}. 
Throughout, an important reference is going to be Meyn and Tweedie's comprehensive book~\cite{MT}. 
\subsection{Preliminaries}\label{subsec:prelim}
We start with preliminaries from general state space Markov chain theory. 
Readers familiar with the subject may skip this subsection and only consult it for notation, when necessary. 

Even though large parts of the theory are valid under the mere assumption that the state space is a measurable space endowed with a countably generated $\sigma$-algebra, for us it is not going to be a restriction to assume that $X$ is a $\sigma$-compact locally compact metrizable space endowed with its Borel $\sigma$-algebra $\B$. 

The first notion to introduce is that of a \emph{transition kernel} on $X$: 
This is a map $P\colon X\times\B\to[0,\infty]$ such that $P(x,\cdot)$ is a Borel measure on $X$ for every $x\in X$ and $x\mapsto P(x,A)$ is measurable for every $A\in \B$. 
It acts on functions $f$ on $X$ from the left and on measures $\nu$ on $X$ from the right by virtue of 
\begin{align*}
Pf(x)=\int_XP(x,\dd y)f(y)\quad\text{and}\quad\nu P(A)=\int_X\nu(\dd x)P(x,A)
\end{align*}
for $x\in X$ and $A\in \B$. 
A transition kernel is called \emph{stochastic} if every $P(x,\cdot)$ is a probability measure, and \emph{substochastic} if $P(x,X)\le 1$ for every $x\in X$. 
A $\sigma$-finite measure $\nu$ on $X$ is called \emph{$P$-subinvariant} if $\nu P\le \nu$ and \emph{$P$-invariant} if $\nu P=\nu$. 
When the transition kernel is clear from context, we just speak of \emph{\textup{(}sub\textup{)}invariant} measures. 
The powers $P^n$ of $P$ are defined inductively by $P^0(x,\cdot)=\delta_x$ and $P^n(x,A)=\int_XP^{n-1}(x,\dd y)P(y,A)$ for $n\in\N$, which generalizes to the \emph{Chapman--Kolmogorov equations} 
\begin{align*}
P^{m+n}(x,A)=\int_XP^m(x,\dd y)P^n(y,A)
\end{align*}
for $x\in X$, $A\in\B$, and $m,n\in\N$. 

A \emph{Markov chain} on $X$ is an $X$-valued stochastic process $\Phi=(\Phi_n)_{n\in\N_0}$ whose steps are governed by a stochastic transition kernel. 
Formally, this means that there exists a starting distribution $\nu$ on $X$ and a stochastic transition kernel $P$ on $X$ such that 
\begin{align*}
\Pro\sqbr{\Phi_0\in A_0,\dots,\Phi_n\in A_n}=\int_{x_0\in A_0}\dotsi\int_{x_{n-1}\in A_{n-1}}\nu(\dd x_0)P(x_0,\dd x_1)\dotsm P(x_{n-1},A_n)
\end{align*}
for every $n\in\N_0$ and $A_0,\dots,A_n\in\B$. 
This formula (specifically, the absence of the variables $x_0,\dots,x_{k-1}$ in the term $P(x_k,\dd x_{k+1})$) captures the quintessential idea behind a Markov chain that the distribution of the following state $\Phi_{n+1}$ depends only on the current state $\Phi_n$ via the transition kernel $P$. 
In terms of conditional distributions, this dependence can be expressed as 
\begin{align*}
\mathcal{L}(\Phi_{n+1}|\Phi_n=x,\Phi_{n-1},\dots,\Phi_0)=\mathcal{L}(\Phi_{n+1}|\Phi_n=x)=P(x,\cdot).
\end{align*}
It may be shown that a Markov chain on $X$ exists for every fixed starting distribution $\nu$ and stochastic transition kernel $P$ (\cite[Theorem~3.4.1]{MT}). 
In fact $\Phi$ may always be assumed to be the canonical coordinate process on $X^{\N_0}$; only the probability measure $\Pro$ on $X^{\N_0}$ needs to be chosen accordingly. 
It is customary to regard the starting distribution as variable and think of a Markov chain on $X$ as being defined by the transition kernel $P$ alone. 
The probability measure on $X^{\N_0}$ making the canonical process into a Markov chain with starting distribution $\nu$ is then denoted by $\Pro_\nu$. 
When $\nu=\delta_x$ is the Dirac mass at some $x\in X$, one simply writes $\Pro_x$. 
The associated expectations are denoted $\E_\nu$ and $\E_x$, respectively. 
\begin{example}\label{ex:RW}
The random walk on $X=G/\Gamma$ given by a probability measure $\mu$ on $G$ is a Markov chain with transition kernel 
\begin{align*}
P(x,\cdot)=\mu*\delta_x.
\end{align*}
Its powers are given by $P^n(x,\cdot)=\mu^{*n}*\delta_x$, where $\mu^{*n}$ is the $n$-th convolution power of $\mu$, defined inductively by $\mu^{*0}=\delta_e$, where $e\in G$ is the identity element, and $\mu^{*n}=\int_Gg_*\mu^{*(n-1)}\dd\mu(g)$ for $n\in\N$. 
Equivalently, $\mu^{*n}$ is the law of a product $Y_n\dotsm Y_1$ of i.i.d.\ random variables $Y_1,\dots,Y_n$ in $G$ with distribution $\mu$. 
If $\mathcal{L}_x$ denotes the law under $\Pro_x$ for some $x\in X$, we thus have 
\begin{align*}
\mathcal{L}_x(\Phi_n)=P^n(x,\cdot)=\delta_xP^n=\mu^{*n}*\delta_x,
\end{align*}
and, more generally, for a starting distribution $\nu$ on $X$, 
\begin{align*}
\mathcal{L}_\nu(\Phi_n)=\nu P^n=\mu^{*n}*\nu.&\qedhere
\end{align*}
\end{example}

Let us next introduce a few important quantities associated to a Markov chain. 
The \emph{first return time} $\tau_A$ and \emph{occupation time} $\eta_A$ of a set $A\in\B$ are defined by 
\begin{align*}
\tau_A=\min\set{n\ge 1\for \Phi_n\in A},\quad\eta_A=\sum_{n=1}^\infty \mathds{1}_{\Phi_n\in A},
\end{align*}
and the return probability and expected number of visits to $A$ starting from $x$ are 
\begin{align*}
L(x,A)=\Pro_x\sqbr{\tau_A<\infty},\quad U(x,A)=\E_x\sqbr{\eta_A}=\sum_{n=1}^\infty P^n(x,A),
\end{align*} 
respectively. 
Note that $U\colon X\times \B\to[0,\infty]$ is a transition kernel on $X$. 

We now address the notion of $\psi$-irreducibility, which was already mentioned in~\S\ref{sec:intro}. 
A $\sigma$-finite measure $\varphi$ on $X$ is called an \emph{irreducibility measure} for a Markov chain on $X$ if for every $A\in\B$ with $\varphi(A)>0$ we have $L(x,A)>0$ for all $x\in X$. 
In other words, this means that any $\varphi$-positive set can be reached from everywhere with positive probability. 
The Markov chain is called \emph{$\psi$-irreducible} if it admits a non-trivial irreducibility measure. 
In this case, it can be shown that there exists a \emph{maximal irreducibility measure}, that is, an irreducibility measure $\psi$ with the property that every other irreducibility measure is absolutely continuous with respect to $\psi$ (\cite[Proposition~4.2.2]{MT}). 
Without loss of generality one may assume $\psi$ to be a probability measure. 
By definition, the measure class of a maximal irreducibility measure is uniquely determined by the Markov chain (i.e.\ by its defining transition kernel $P$). 
This justifies the implicit understanding (and slight abuse of notation) common in the literature that, given a $\psi$-irreducible Markov chain, $\psi$ always denotes an associated maximal irreducibility measure. 

For $\psi$-irreducible chains there is a recurrence/transience dichotomy similar to the classical discrete theory. 
To state it, we call a set $A\subset X$ \emph{uniformly transient} if the expected number of returns to $A$ is bounded on $A$, i.e.\ if $\sup_{x\in A}U(x,A)<\infty$, and \emph{recurrent} if the expected number of returns is infinite on all of $A$, i.e.\ if $U(x,A)=\infty$ for all $x\in A$. 
\begin{theorem}[{\cite[Theorem~8.0.1]{MT}}]\label{thm:recurrence}
Suppose $\Phi$ is $\psi$-irreducible. 
Then either 
\begin{enumerate}
\item every $\psi$-positive set is recurrent, in which case $\Phi$ is called \emph{recurrent}, or 
\item the state space $X$ can be covered by countably many uniformly transient sets, in which case $\Phi$ is called \emph{transient}. 
\end{enumerate}
\end{theorem}
We emphasize that $\psi$-irreducibility is included in these definitions of recurrence and transience. 
For recurrent chains, one has the following conclusion about invariant measures. 
\begin{theorem}[{\cite[Theorem~10.4.9]{MT}}]\label{thm:inv_meas}
Suppose $\Phi$ is recurrent. 
Then there exists a non-trivial $\sigma$-finite invariant measure $\pi$, which is unique up to scalar multiples. 
Moreover, $\pi$ is a maximal irreducibility measure. 
\end{theorem}
As in the classical theory, a further refinement of recurrence is possible: 
The chain is called \emph{positive} if it is $\psi$-irreducible and admits a non-trivial finite invariant measure. 
This forces the chain to be recurrent. 
\begin{proposition}[{\cite[Proposition~10.1.1]{MT}}]\label{prop:positive}
A positive chain is recurrent. 
In particular, a positive chain admits a unique invariant probability measure, which is a maximal irreducibility measure. 
\end{proposition}
For this reason, positive chains are also called \emph{positive recurrent}. 

In the general theory, there is one more important notion of recurrence that does not appear in the discrete theory. 
Namely, in the latter, a recurrent state $x$ always satisfies $\Pro_x\sqbr{\tau_x<\infty}=1$, and hence by the Markov property also $\Pro_x\sqbr{\eta_x=\infty}=1$. 
Since in more general spaces there might be no returns to the precise starting point, such conclusions can no longer be drawn. 
Let us write 
\begin{align*}
Q(x,A)=\Pro_x\sqbr{\eta_A=\infty}
\end{align*}
for $x\in X$ and $A\in\B$, call the set $A$ \emph{Harris recurrent} if $Q(x,A)=1$ for every $x\in A$, and the whole chain $\Phi$ \emph{Harris recurrent} if it is $\psi$-irreducible and every $\psi$-positive set is Harris recurrent. 
Clearly, Harris recurrence implies recurrence. 
We call $\Phi$ \emph{positive Harris recurrent} if it is positive and Harris recurrent. 

The final notion we need to introduce is that of aperiodicity, which naturally plays a role in questions of convergence to a stable distribution. 
\begin{theorem}[{\cite[Theorem~5.4.4]{MT}}]\label{thm:markov_periodicity}
Let $\Phi$ be $\psi$-irreducible. 
Then there exists a maximal positive integer $d$, called the \emph{period} of $\Phi$, with the property that there exist pairwise disjoint sets $D_0,\dots,D_{d-1}\in\B$ such that $P(x,D_{i+1\bmod d})=1$ for each $x\in D_i$ and $i=0,\dots,d-1$ and such that the union $\bigcup_{i=0}^{d-1} D_i$ is $\psi$-full. 
\end{theorem}
A collection of measurable sets $D_0,\dots,D_{d-1}$ as in the above theorem is referred to as a \emph{$d$-cycle} for $\Phi$. 
A $\psi$-irreducible chain with period $1$ is called \emph{aperiodic}. 
\subsection{\texorpdfstring{$T$}{T}-Chains}\label{subsec:T_chains}
As already pointed out, the notions of recurrence for $\psi$-irreducible chains require certain properties of returns to $\psi$-positive measurable sets from arbitrary starting points, not taking into account topological properties of the state space. 
Of course this makes sense, as the topology did not feature in any of the definitions up to this point. 
In order to connect the chain to the topology, one thus needs an additional concept. 
Several notions accomplishing this appear in the literature; the one best suited for the study of random walks is that of a \emph{$T$-chain} introduced by Tuominen--Tweedie~\cite{TT}. 
Its definition involves the \emph{sampling} of a transition kernel $P$: 
Given a probability distribution $a$ on $\N_0$, the \emph{sampled transition kernel $K_a$} with \emph{sampling distribution $a$} is defined by 
\begin{align*}
K_a=\sum_{n=0}^\infty a(n)P^n.
\end{align*}
\begin{definition}
A Markov chain $\Phi$ on $X$ given by a transition kernel $P$ is called a \emph{$T$-chain} if there exists a sampling distribution $a$ on $\N_0$ and a substochastic transition kernel $T$ on $X$ with 
\begin{enumerate}
\item $K_a(x,A)\ge T(x,A)$ for all $x\in X$ and $A\in\B$, 
\item $T(x,X)>0$ for all $x\in X$, and such that 
\item $T(\cdot,A)$ is lower semicontinuous for all $A\in \B$. 
\end{enumerate}
We call $T$ a \emph{continuous component of $P$}. 
\end{definition}
Let us outline the links the $T$-property establishes between recurrence and topology. 
We call a state $x_0\in X$ \emph{reachable} if $L(x,B)>0$ for every $x\in X$ and neighborhood $B$ of $x_0$, \emph{topologically Harris recurrent} if $Q(x_0,B)=1$ for each neighborhood $B$ of $x_0$, and \emph{topologically recurrent} if $U(x_0,B)=\infty$ for each neighborhood of $x_0$. 
If $x_0$ is not topologically recurrent, it is called \emph{topologically transient}. 
The first result we shall need infers $\psi$-irreducibility from the existence of a reachable state. 
\begin{proposition}[{\cite[Proposition~6.2.1]{MT}}]\label{prop:reachable-irreducible}
If a $T$-chain admits a reachable state, it is $\psi$-irreducible. 
\end{proposition}
The second one is a strong decomposition statement, allowing the splitting of the state space $X$ into a Harris recurrent and a transient part. 
\begin{theorem}[{\cite[Theorem~9.3.6]{MT}}]\label{thm:T_decomposition}
For a $\psi$-irreducible $T$-chain, the state space $X$ admits a decomposition 
\begin{align*}
X=H\sqcup N
\end{align*}
into a \emph{Harris set} $H$ \textup{(}meaning that $P(x,H)=1$ for each $x\in H$ and the restriction of the chain to $H$ is Harris recurrent\textup{)} and a set $N$ consisting of topologically transient states. 
\end{theorem}
The following result is the motivation for introducing Markov chain methods in the study of spread out random walks. 
For random walks on groups it is due to Tuominen--Tweedie \cite[Theorem~5.1(i)]{TT}; the case of a homogeneous space $X$ is not much more complicated. 
\begin{proposition}\label{prop:T_equivalence}
Let $G$ be a $\sigma$-compact locally compact metrizable group, $\Gamma\subset G$ a discrete subgroup, and $X$ the homogeneous space $G/\Gamma$. 
Then the random walk on $X$ given by a probability measure $\mu$ on $G$ is a $T$-chain if and only if $\mu$ is spread out. 
In this case, the sampling distribution $a$ and the continuous component $T$ may be chosen such that 
\begin{itemize}
\item $a=\delta_{n_0}$ for some $n_0\in\N$, 
\item $T(\cdot,X)$ is constant, and 
\item $Tf$ is continuous for every bounded measurable function $f$ on $X$. 
\end{itemize}
\end{proposition}
For convenience we include a proof, which adapts that of \cite[Proposition~6.3.2]{MT} to the setting at hand. 
\begin{proof}
Denote by $\pr\colon G\to X,\,g\mapsto g\Gamma,$ the canonical projection. 
Recalling Example~\ref{ex:RW}, we see that the powers of the transition kernel $P$ of the random walk on $X$ are given by 
\begin{align}\label{concrete_kernel}
P^n(x,A)=\int_Gh_*\delta_{g\Gamma}(A)\dd\mu^{*n}(h)=\mu^{*n}(\pr^{-1}(A)g^{-1})
\end{align}
for $n\in\N$, $x=g\Gamma\in X$, and $A\subset X$. 
Let $m_G$ denote a left Haar measure on $G$. 

Assume first that the random walk is a $T$-chain. 
If every convolution power $\mu^{*n}$ for $n\in\N$ is singular with respect to $m_G$, we find a set $E_G\subset G$ with $\mu^{*n}(E_G)=1$ for all $n\in\N$ and $m_G(E_G)=0$. 
Enlarging $E_G$ if necessary, we may assume that the identity $e\in G$ belongs to $E_G$ 
and that $E_G$ is right-$\Gamma$-invariant. 
Write $E=\pr(E_G)$ and let $a$ be the sampling distribution associated to the continuous component $T$ of the random walk. 
Then 
\begin{align*}
T(e\Gamma,E^c)\le K_a(e\Gamma,E^c)=\sum_{n=0}^\infty a(n)\underbrace{P^n(e\Gamma,E^c)}_{=\mu^{*n}(E_G^c)=0}=0,
\end{align*}
where we used \eqref{concrete_kernel}, that $\pr^{-1}(E^c)=E_G^c$ by the assumed right-$\Gamma$-invariance and $e\in E_G$ for $n=0$. 
Properties (ii) and (iii) in the definition of a $T$-chain thus produce $\delta>0$ and a neighborhood $B$ of $e\Gamma\in X$ with $T(x,E)\ge \delta$, and hence also 
\begin{align*}
K_a(x,E)\ge \delta
\end{align*}
for all $x\in B$. 
But by translation invariance of $m_G$ and Fubini's theorem, we find 
\begin{align*}
m_G(E_G)&=\int_Gm_G(g^{-1}E_G)\dd\mu^{*n}(g)\\
&=\int_G\mu^{*n}(E_Gh^{-1})\dd m_G(h)\\
&=\int_GP^n(h\Gamma,E)\dd m_G(h),
\end{align*}
which, after summing with the weights $a(n)$, yields the contradiction 
\begin{align*}
m_G(E_G)&=\int_GK_a(h\Gamma,E)\dd m_G(h)\\
&\ge \int_{\pr^{-1}(B)}K_a(h\Gamma,E)\dd m_G(h)\\
&\ge \delta m_G(\pr^{-1}(B))>0.
\end{align*}

For the converse, suppose that $\mu^{*n_0}$ is not singular with respect to $m_G$ for some $n_0\in\N$. 
Then there exists a non-negative $m_G$-integrable function $p\colon G\to\R$ with $\int p\dd m_G>0$ and $\dd\mu^{*n_0}\ge p\dd m_G$. 
Denoting by $\Delta$ the modular character of $G$, we obtain for $x=g\Gamma\in X$ and $A\subset X$ 
\begin{align*}
P^{n_0}(x,A)\ge\int_{\pr^{-1}(A)g^{-1}}p\dd m_G = \Delta(g)^{-1}\int_{\pr^{-1}(A)}p(g'g^{-1})\dd m_G(g')\eqqcolon T(x,A).
\end{align*}
The sampling distribution $a=\delta_{n_0}$ together with this $T$ are then seen to possess all claimed properties. 
\end{proof}
\section{Spread Out Random Walks}\label{sec:spread_out}
This section is the central part of the paper, aiming to give a complete picture of the qualitative behavior of spread out random walks on homogeneous spaces. 

In what follows, we are not going to assume that $\Gamma$ is a lattice or that $G$ is unimodular, so that there will in general be no $G$-invariant measure on the quotient $X=G/\Gamma$. 
However, for every continuous character $\chi\colon G\to \R_{>0}$ extending the restriction $\Delta|_\Gamma$ of the modular character $\Delta$ of $G$ to $\Gamma$, there exists a non-trivial Radon measure $m_{X,\chi}$ on $X$ that is \emph{$\chi$-quasi-invariant} in the sense that 
\begin{align*}
g_*m_{X,\chi}=\chi(g)m_{X,\chi}
\end{align*}
for all $g\in G$. 
Such a measure is unique up to scalars. 
Two important cases of this construction are $\chi=\Delta$, the choice of which is always possible, and $\chi=\mathds{1}$, which is a possible choice whenever $\Delta(\gamma)=1$ for all $\gamma\in\Gamma$. 
In the latter case, $m_X=m_{X,\mathds{1}}$ is a Haar measure on $X$. 
All $m_{X,\chi}$ belong to the same measure class, which we refer to as the \emph{Haar measure class} on $X$. 
This terminology is justified by the fact that $m_{X,\Delta}$ can be identified with the restriction of a right Haar measure on $G$ to a fundamental domain for $\Gamma$. 
We refer to \cite[Ch.VII\S2]{bourbaki} for details. 

Slightly abusing notation, we are going to denote the Haar measure class on $X$ by $[m_X]$, and for a measure $\nu$ on $X$ write $\nu\ll[m_X]$, $\nu\sim[m_X]$, $[m_X]\ll \nu$ to express that $\nu$ is absolutely continuous with respect to $[m_X]$, contained in $[m_X]$, or that $[m_X]$ is absolutely continuous with respect to $\nu$, respectively. 

Let us summarize at this point the standing assumptions and notations that will be in effect for the remainder of the article when nothing else is specified. 
\begin{center}
\fbox{
\begin{minipage}{.8\textwidth}
\textsc{Standing Assumptions:} 
$\mu$ is a probability measure on a locally compact $\sigma$-compact metrizable group $G$; $\calS$ and $\G$ are the closed subsemigroup and subgroup of $G$ generated by $\supp(\mu)$, respectively; $\Gamma\subset G$ is a discrete subgroup; $X$ is the homogeneous space $G/\Gamma$; $[m_X]$ is the Haar measure class on $X$ and $m_X$ a Haar measure (when one exists); and $P$ is the transition kernel of the random walk on $X$ induced by $\mu$. 
\end{minipage}
}
\end{center}

\subsection{Transitivity \& \texorpdfstring{$\psi$}{psi}-Irreducibility}\label{subsec:trans}
Let $x\in X$ be the starting point for our random walk. 
Then, in some sense, everything outside the closed subgroup $\G$ of $G$ generated by $\supp(\mu)$ and outside the orbit $\G x\subset X$ is irrelevant for the study of the random walk. 
The following simple lemma shows how such redundancy can be removed. 
\begin{lemma}\label{lem:adapted}
Let $\mu$ be spread out. 
Then $\G$ is an open subgroup of $G$. 
For every $x=g\Gamma\in X$ the orbit $\G x$ is a clopen subset of $X$ satisfying $\G/(\G\cap g\Gamma g^{-1})\cong \G x$. 
If $X$ has finite volume, then so does $\G/(\G\cap g\Gamma g^{-1})$. 
\end{lemma}
\begin{proof}
From the formula 
\begin{align}\label{convolution_support}
\overline{\supp(\mu^{*m})\supp(\mu^{*n})}=\supp(\mu^{*(m+n)})
\end{align}
for $m,n\in\N$ we see that $\supp(\mu^{*n})\subset \G$ for every $n\in\N$. 
Since $\mu$ is spread out and the convolution of bounded integrable functions on $G$ is continuous, some convolution power $\mu^{*n_0}$ has a component with continuous density with respect to Haar measure on $G$. 
Thus $\G\supset\supp(\mu^{*n_0})$ has non-empty interior, and consequently $\G$ is open. 
Since the action map $G\ni g\mapsto gx\in X$ is a local homeomorphism, this implies that also $\G x$ is open. 
But then $X$ is a disjoint union of such open $\G$-orbits, so that all of them must also be closed. 
Writing $x=g\Gamma$, the isomorphism $\G/(\G\cap g\Gamma g^{-1})\cong \G x$ of $\G$-spaces follows, since $\G\cap g\Gamma g^{-1}=\operatorname{Stab}_\G(x)$. 
When $X$ has finite volume, this quotient supports a finite invariant measure inherited from the restriction of Haar measure on $X$ to $\G x$, so that $\G\cap g\Gamma g^{-1}$ is a lattice in $\G$. 
\end{proof}
In other words, at the price of replacing $X$ by $\G x$, we are free to assume that $\mu$ is adapted. 
In view of this, we will formulate most of the following results only for adapted measures. 

Preparing for the proof of $\psi$-irreducibility of spread out random walks, our next objective is to find a more efficient description of an orbit $\G x$. 
We will use the notation $\Gamma_g=g\Gamma g^{-1}$ for $g\in G$. 
\begin{lemma}\label{lem:semigroup_transitive}
The set $S=\bigcup_{n=1}^\infty \supp(\mu^{*n})$ is a subsemigroup of $G$ with $\overline{S}=\calS$. 
If $\mu$ is spread out and adapted and $\calS \Gamma_g=\set{s\gamma\for s\in \calS,\,\gamma\in\Gamma_g}$ equals $G$ for all $g\in G$, then $S$ acts transitively on $X$. 
\end{lemma}
\begin{proof}
That $S$ is a semigroup with $\supp(\mu)\subset S\subset \calS$ follows from \eqref{convolution_support}. 
Since $\calS$ is by definition the smallest closed subsemigroup of $G$ containing $\supp(\mu)$, we must have $\overline{S}=\calS$. 

Let us now show transitivity of the $S$-action on $X$ under the stated assumptions. 
To this end, note first that $S\Gamma_g$ is dense in $G$ for every $g\in G$, since 
\begin{align*}
\overline{S\Gamma_g}=\overline{\bigcup_{\gamma\in\Gamma_g}S\gamma}\supset\bigcup_{\gamma\in\Gamma_g}\overline{S\gamma}=\bigcup_{\gamma\in\Gamma_g}\calS\gamma=\calS\Gamma_g=G.
\end{align*}
Now let $x,y\in X$ be arbitrary. 
We need to find an element of $S$ sending $x$ to $y$. 
Choose $g\in G$ with $gx=y$ and write $x=h\Gamma$ for some $h\in G$. 
Using that $S$ has non-empty interior (by the same argument as in Lemma~\ref{lem:adapted}), we can find a non-empty open subset $U$ of $G$ contained in $S$. 
By density of $S\Gamma_h$ in $G$, it follows that $U^{-1}g$ intersects $S\Gamma_h$ non-trivially, say $u^{-1}g=sh\gamma h^{-1}$ for some $u\in U\subset S$, $s\in S$ and $\gamma\in\Gamma$. 
Recalling that $x=h\Gamma$, we conclude that
\begin{align*}
usx=gh\gamma^{-1} h^{-1}x=gh\Gamma=gx=y,
\end{align*}
so that the element $us\in S$ has the required property. 
\end{proof}

The conclusion of the previous lemma will be important for many of the following results. 
Let us therefore give a name to its set of assumptions. 
\begin{definition}
We say that a probability measure $\mu$ on $G$ is \emph{$\Gamma$-adapted} if $\mu$ is adapted and $\calS \Gamma_g=G$ for all $g\in G$, where $\Gamma_g=g\Gamma g^{-1}$. 
\end{definition}

For spread out random walks on finite volume spaces, the second requirement in the above definition is redundant. 
\begin{proposition}\label{prop:fin_vol_adapted}
Let $\mu$ be spread out and adapted and suppose that $X$ has finite volume. 
Then $\mu$ is $\Gamma$-adapted. 
\end{proposition}
\begin{proof}
We claim that for every $x\in X$, the orbit $A=\calS x$ equals $X$. 
This will imply that for every $g,g'\in G$ there exists $s\in\calS$ with $sg\Gamma=g'g\Gamma$, which is the desired conclusion. 

To prove the claim, observe that $A$ satisfies $sA\subset A\subset s^{-1}A$ for every $s\in \calS$. 
By invariance of $m_X$ we also know that the $m_X$-measures of these three sets coincide, so it follows that the characteristic function $\mathds{1}_A$ is $m_X$-a.s.\ invariant under each element of $\calS\cup \calS^{-1}$ (individually). 
We conclude that $\mathds{1}_A$ is $m_X$-a.s.\ invariant under each element of a dense subset of $G$, hence under all of $G$ by continuity of the regular representation on $L^1(X)$. 
But as $\calS$ has non-empty interior, we know that $A$ has positive measure, so that $G$-invariance forces $m_X(A)=1$. 
But then, if there was some $y\in X\setminus A$, we would have a set $\calS^{-1}y$ disjoint from $A$ which also has positive measure (since also $\calS^{-1}$ has non-empty interior), which is a contradiction. 
\end{proof}

We can now relate the property of a probability measure being spread out to $\psi$-irreducibility of the induced random walk. 
Recall the convention that when speaking about $\psi$-irreducibility, $\psi$ always denotes a maximal irreducibility measure. 
\begin{proposition}\label{prop:spread_out_irr}
Let $\mu$ be a probability measure on $G$. 
If the random walk on $X$ given by $\mu$ is $\psi$-irreducible, then $\mu$ is spread out and $\psi\ll [m_X]$. 
Conversely, if $\mu$ is spread out and $\Gamma$-adapted, then the random walk on $X$ is $\psi$-irreducible with $[m_X]\sim\psi$. 
\end{proposition}
\begin{proof}
Suppose first that the random walk on $X$ given by $\mu$ is $\psi$-irreducible and let $m_{X,\chi}$ be a quasi-invariant measure on $X$ for some character $\chi$ of $G$. 
Then for every measurable subset $A\subset X$ we find, using Fubini's theorem, 
\begin{align}\label{quasi_stat}
\int_XP(x,A)\dd m_{X,\chi}(x)=\int_Gm_{X,\chi}(h^{-1}A)\dd\mu(h)=\int_G\chi\dd\mu \cdot m_{X,\chi}(A).
\end{align}
In other words, we have $m_{X,\chi}P=c(\chi,\mu)m_{X,\chi}$ for the constant $c(\chi,\mu)=\int_G\chi\dd\mu$. 
Consider the sampled transition kernel $K_a=\sum_{n\ge 0}a(n)P^n$ with $a(n)=2^{-(n+1)}$ for $n\in\N_0$. 
By definition of an irreducibility measure, for every $\psi$-positive set $A\subset X$ it satisfies $K_a(x,A)>0$ for all $x\in X$. 
If $A$ were an $[m_X]$--null set, it would follow that
\begin{align*}
0<m_{X,\chi}K_a(A)=\sum_{n=0}^\infty a(n)m_{X,\chi}P^n(A)=\sum_{n=0}^\infty a(n)c(\chi,\mu)^nm_{X,\chi}(A)=0,
\end{align*}
which is a contradiction. 
We have thus shown that $\psi\ll [m_X]$. 
If $\mu$ is not spread out, then as in the proof of Proposition~\ref{prop:T_equivalence} there exists a right-$\Gamma$-invariant measurable set $e\in E_G\subset G$ with $\mu^{*n}(E_G)=1$ for all $n\in\N$ and $m_G(E_G)=0$, where $m_G$ denotes a left Haar measure on $G$. 
The set $E=\pr(E_G)$, where $\pr\colon G\to X,\,g\mapsto g\Gamma,$ denotes the projection, is then an $[m_X]$--null set. 
Using $\psi\ll [m_X]$ it follows that $\psi(E^c)>0$; yet we have $P^n(e\Gamma,E^c)=\mu^{*n}(E_G^c)=0$ for all $n\in\N_0$. 
This contradicts $\psi$-irreducibility, hence $\mu$ must be spread out. 

For the converse, recall from Proposition~\ref{prop:T_equivalence} that the random walk on $X$ induced by a spread out measure $\mu$ is a $T$-chain. 
By Proposition~\ref{prop:reachable-irreducible}, $\psi$-irreducibility can be established by proving existence of a reachable state. 
But from Lemma~\ref{lem:semigroup_transitive} it in fact follows that every $x_0\in X$ is reachable: 
Given any other point $x\in X$, it can be written as $x_0=sx$ for some $s\in \supp(\mu^{*n})$, and we conclude for any neighborhood $B$ of the identity in $G$ that 
\begin{align*}
L(x,Bx_0)\ge P^n(x,Bx_0)=P^n(x,Bsx)\ge \mu^{*n}(Bs)>0.
\end{align*}
Hence, the random walk is $\psi$-irreducible. 
The first part of the proposition thus yields $\psi\ll[m_X]$. 
To also obtain $[m_X]\ll\psi$, it suffices to show that members of the Haar measure class are irreducibility measures. 
Let therefore $A\subset X$ be an $[m_X]$-positive set and define $A_G=\pr^{-1}(A)$. 
Then also $m_G(A_G)>0$. 
By Proposition~\ref{prop:T_equivalence} and its proof, for some $n_0\in\N$ the kernel $P^{n_0}$ has a continuous component $T$ given by an absolutely continuous measure $p\dd m_G$ on $G$, where $p$ is an $m_G$-integrable function on $G$ with $\int_Gp\dd m_G>0$. 
In particular, we know $m_G(p^{-1}((0,\infty)))>0$, so that by a standard fact of measure theory also $m_G(A_Gg^{-1}\cap p^{-1}((0,\infty)))>0$ for some $g\in G$. 
It follows that 
\begin{align}\label{T_irred}
T(g\Gamma,A)=\int_{A_Gg^{-1}}p\dd m_G>0.
\end{align}
But as $g\Gamma$ is reachable, \cite[Proposition~6.2.1]{MT} implies that $T(g\Gamma,\cdot)$ is an irreducibility measure, so that \eqref{T_irred} entails $L(x,A)>0$ for all $x\in X$. 
This completes the proof. 
\end{proof}
\begin{corollary}\label{cor:fin_vol_irred}
Let $\mu$ be spread out and adapted and suppose that $X$ has finite volume. 
Then the random walk on $X$ given by $\mu$ is $\psi$-irreducible with $\psi\sim[m_X]$. 
\end{corollary}
\begin{proof}
Combine Propositions~\ref{prop:fin_vol_adapted} and~\ref{prop:spread_out_irr}. 
\end{proof}
One may wonder if in the first statement of the previous proposition, $\psi$ must even belong to the Haar measure class on $X$. 
In view of the second conclusion this is true when $\mu$ is additionally $\Gamma$-adapted. 
In general however, it does not hold, as the following example demonstrates. 
\begin{example}
Let $G=\R_{>0}\ltimes \R$ be the $ax+b$-group of affine transformations of $\R$, with group law given by
\begin{align*}
(a,b)(a',b')=(aa',ab'+b)
\end{align*}
for $a,a'\in\R_{>0}$ and $b,b'\in \R$, and consider the discrete subgroup $\Gamma$ of $G$ given by $\Gamma=\set{2^n\for n\in\Z}\times\set{0}$. 
We decompose $G$ into $G^+=\set{(a,b)\in G\for b\ge 0}$ and $G^-=\set{(a,b)\in G\for b<0}$ and define $X^\pm=\pr(G^\pm)$, where $\pr\colon G\to X=G/\Gamma$ denotes the projection. 
Our goal is to construct a $\psi$-irreducible random walk on $X$ which never moves from $X^+$ to $X^-$. 
Then it will follow that any irreducibility measure for this random walk must have support inside $X^+$ and thus cannot belong to the Haar measure class. 

The following construction achieves this goal. 
Let $\mu$ be a probability measure on $G$ absolutely continuous with respect to a right Haar measure $m_G$ with a density that is strictly positive on $(0,1)\times \R_{\ge 0}$ and $0$ otherwise. 
For example, one may choose $\dd\mu(a,b)=\mathds{1}_{(0,1)\times \R_{\ge 0}}(a,b)\euler^{-b}\dd a\dd b$. 
Let $m_{X,\Delta}$ be the quasi-invariant measure on $X$ coming from the modular character $\Delta$ of $G$. 
Then given a starting point $x\in X$, the law $\mathcal{L}_x(\Phi_1)=\mu*\delta_x$ after the first step of the random walk is absolutely continuous with respect to $m_{X,\Delta}$. 
By choice of $\mu$ and definition of the group operation, the corresponding density $p_x=\frac{\dd(\mu*\delta_x)}{\dd m_{X,\Delta}}$ is seen to have the following properties: 
\begin{itemize}
\item Irrespective of the location of the starting point $x$, this density $p_x$ is strictly positive almost everywhere on $X^+$. 
\item For $x\in X^+$, $p_x$ is $0$ on $X^-$. 
\end{itemize}
Indeed, both properties follow from the geometry of the left action of $G$ on $X$, which can be understood e.g.\ by identifying $X$ with the fundamental domain $F=[1,2)\times \R$ for $\Gamma$ inside $G$; see Figure~\ref{fig:example}. 
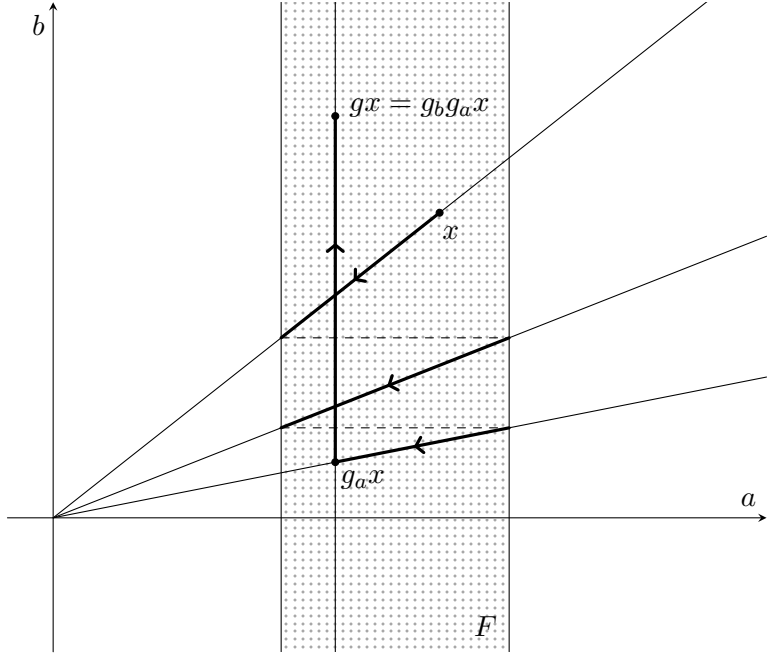
\begin{figure}[ht]
\centering
\begin{tikzpicture}[line cap=round,line join=round,>=triangle 45,x=1.0cm,y=1.0cm]
\begin{axis}[
ticks=none,
x=3cm,y=1.0cm,
axis lines=middle,
xmin=-0.20054412865667584,
xmax=3.129019977647152,
ymin=-1.7765263726431242,
ymax=6.835719848227108,
xtick={},
ytick={-1.0,0.0,...,6.0},]
\clip(-0.20054412865667584,-1.7765263726431242) rectangle (3.129019977647152,6.835719848227108);
\fill[line width=0.pt,color=wqwqwq,fill=wqwqwq,pattern=dots,pattern color=xdxdff] (1.,7.2402304031231735) -- (2.,7.2402304031231735) -- (2.,-5.551014218805338) -- (1.,-5.532286188612177) -- cycle;
\draw  (2.,-1.7765263726431242) -- (2.,6.835719848227108);
\draw  (1.,-1.7765263726431242) -- (1.,6.835719848227108);
\draw [domain=0.0:3.129019977647152] plot(\x,{(-0.--4.044714128500644*\x)/1.6949236593430783});
\draw [domain=0.0:3.129019977647152] plot(\x,{(-0.--2.3863695017793907*\x)/2.});
\draw [domain=0.0:3.129019977647152] plot(\x,{(-0.--1.1931847508896953*\x)/2.});
\draw [dash pattern=on 3pt off 3pt] (1.,2.3863695017793907)-- (2.,2.3863695017793907);
\draw [dash pattern=on 3pt off 3pt] (1.,1.1931847508896953)-- (2.,1.1931847508896953);
\draw [line width=1.2pt] (1.6949236593430783,4.044714128500644)-- (1.,2.3863695017793907);
\draw [line width=1.2pt] (1.325004214442977,3.1619496470758803) -- (1.3257569481428253,3.295390917048339);
\draw [line width=1.2pt] (1.325004214442977,3.1619496470758803) -- (1.3691667112002532,3.1356927132316956);
\draw [line width=1.2pt] (2.,2.3863695017793907)-- (1.,1.1931847508896953);
\draw [line width=1.2pt] (1.4732516928620414,1.7578614541454167) -- (1.4870741209248108,1.8848819841209363);
\draw [line width=1.2pt] (1.4732516928620414,1.7578614541454167) -- (1.5129258790751894,1.6946722685481497);
\draw [line width=1.2pt] (2.,1.1931847508896953)-- (1.2374677584502027,0.7382638295502174);
\draw [line width=1.2pt] (1.5904654975358419,0.9488595892379798) -- (1.6119036585034296,1.0662338442093868);
\draw [line width=1.2pt] (1.5904654975358419,0.9488595892379798) -- (1.6255640999467735,0.865214736230527);
\draw  (1.2374677584502027,-1.7765263726431242) -- (1.2374677584502027,6.835719848227108);
\draw [line width=1.2pt] (1.2374677584502027,0.7382638295502174)-- (1.2374677584502027,5.326422382943428);
\draw [line width=1.2pt] (1.2374677584502027,3.617782056850694) -- (1.2720706186773787,3.5323431062468225);
\draw [line width=1.2pt] (1.2374677584502027,3.617782056850694) -- (1.2028648982230268,3.5323431062468225);
\begin{scriptsize}
\draw [fill=black] (1.6949236593430783,4.044714128500644) circle (1.3pt);
\draw[color=black] (1.7449055685601103,3.777005416608514) node {$x$};
\draw [fill=black] (1.2374677584502027,0.7382638295502174) circle (1.3pt);
\draw[color=black] (1.3585250534729458,0.501845643460112) node {$g_ax$};
\draw [fill=black] (1.2374677584502027,5.326422382943428) circle (1.3pt);
\draw[color=black] (1.606972675947586,5.434521058323618) node {$gx=g_bg_ax$};
\draw[color=black] (1.896972675947586,-1.434521058323618) node {$F$};
\draw[color=black] (3.05,0.234521058323618) node {$a$};
\draw[color=black] (-.063972675947586,6.534521058323618) node {$b$};
\end{scriptsize}
\end{axis}
\end{tikzpicture}
\caption{Illustration of the left action of $G$ on $X$ in the fundamental domain $F$. The acting element $g=(a,b)\in(0,1)\times\R_{\ge 0}$ is decomposed as $g=g_bg_a$ for $g_a=(a,0)$ and $g_b=(1,b)$. The dashed lines indicate identifications using the right action of $\Gamma$.}
\label{fig:example}
\end{figure}
We deduce that $L(x,X^-)=0$ for all $x\in X^+$, and $L(x,A)>0$ for all $A\subset X$ intersecting $X^+$ in a positive measure set and all $x\in X$. 
Hence, the random walk is $\psi$-irreducible, with maximal irreducibility measure being given e.g.\ by $\psi=m_{X,\Delta}|_{X^+}$. 
\end{example} 

Let us record a situation in which we do not need adaptedness to guarantee that the random walk is $\psi$-irreducible on all of $X$. 
\begin{corollary}\label{cor:full_orbit}
Let $\mu$ be spread out. 
Suppose that $X$ is connected and that 
\begin{itemize}
\item $X$ has finite volume, or that 
\item $\G=\calS$. 
\end{itemize}
Then the random walk on $X$ given by $\mu$ is $\psi$-irreducible with $\psi\sim[m_X]$ and the semigroup $S$ from Lemma~\textup{\ref{lem:semigroup_transitive}} acts transitively on $X$. 
\end{corollary}
\begin{proof}
Lemma~\ref{lem:adapted} implies that $\G x$ is clopen, hence equal to $X$ by connectedness. 
The same lemma thus allows us to assume that $\mu$ is adapted without changing $X$. 
Then, using Proposition~\ref{prop:fin_vol_adapted} in the finite volume case, we see that $\mu$ is $\Gamma$-adapted. 
Lemma~\ref{lem:semigroup_transitive} and Proposition~\ref{prop:spread_out_irr} now give all conclusions. 
\end{proof}
\subsection{Periodicity}
Proposition~\ref{prop:spread_out_irr} states that for (reasonably nice) spread out random walks we have at our disposal the whole theory of $\psi$-irreducible Markov chains from~\S\ref{sec:markov_theory}. 
In particular, we know from Theorem~\ref{thm:markov_periodicity} that they have a well-defined period $d\in\N$. 
Let us look at the sets $D_i$ in a corresponding $d$-cycle in more detail. 
\begin{proposition}\label{prop:periodicity}
Let $\mu$ be a probability measure on $G$. 
Suppose that the random walk on $X$ given by $\mu$ is $\psi$-irreducible with $\psi\sim[m_X]$ and let $d\in\N$ be its period. 
Then there exist subsets $D_0,\dots,D_{d-1}$ of $X$ with the following properties: 
\begin{enumerate}
\item The $D_i$ are clopen, non-empty, and form a partition of $X$, 
\item we have $P(x,D_{i+1\bmod d})=1$ for every $x\in D_i$ and $gD_i=D_{i+1\bmod d}$ for every $g\in\supp(\mu)$, 
\item if $\G_d$ denotes the closed subgroup of $G$ generated by $\supp(\mu^{*d})$, then for every $x\in D_i$ we have $D_i=\G_d x$, and 
\item the $d$-step random walk on each $D_i$ is $\psi$-irreducible and aperiodic, 
\end{enumerate}
where always $i=0,\dots,d-1$. 
\end{proposition}
In other words, a general spread out random walk governed by $\mu$ splits up into $d$ aperiodic spread out random walks governed by the $d$-fold convolution power $\mu^{*d}$. 
\begin{proof}
Throughout the proof, the terms \enquote{null set} or \enquote{full measure set} are understood with respect to the Haar measure class on $X$, to which $\psi$ belongs by assumption. 
We shall make repeated use of the fact that open null sets are empty. 
Moreover, for ease of notation, we will drop the specifier \enquote{$\bmod d$} from the indices, implicitly viewing them as elements of $\Z/d\Z$. 

Let $D_0',\dots,D_{d-1}'\subset X$ be a $d$-cycle as in Theorem~\ref{thm:markov_periodicity}. 
Proposition~\ref{prop:spread_out_irr} shows that $\mu$ is spread out. 
Thus, by Proposition~\ref{prop:T_equivalence}, for some $n_0\in \N$ there is a continuous component $T$ of $P^{n_0}$ with the property that $T\mathds{1}_X$ is constant, say $T\mathds{1}_X\equiv \alpha\in(0,1]$, and $T\mathds{1}_{D_i'}\colon X\to[0,\alpha]$ is continuous for $i=0,\dots,d-1$. 
By the properties of a $d$-cycle there exists a cyclic permutation $\sigma$ of $\set{0,\dots,d-1}$ such that $P^{n_0}\mathds{1}_{D_{\sigma(i)}'}$ is $1$ on $D_i'$ and $0$ on $\bigcup_{j\neq i}D_j'$. 
Together with the above this implies that $f_i=T\mathds{1}_{D_{\sigma(i)}'}$ is $\alpha$ on $D_i'$ and $0$ on $\bigcup_{j\neq i}D_j'$. 
The claim is that the sets 
\begin{align*}
D_i=f_i^{-1}(\set{\alpha})
\end{align*}
have the desired properties. 
Indeed, by construction we know that on each fixed set $D_i'$ the function $f_i$ is $\alpha$ and all other $f_j$ are $0$. 
In particular, the sets $f_i^{-1}((0,\alpha))$ are contained in the complement of the full measure set $\bigcup_{j=0}^{d-1} D_j'$. 
Being open by continuity, they must thus be empty. 
This means that the $f_i$ are in fact continuous maps from $X$ to the discrete space $\set{0,\alpha}$. 
The sets $P_v=\set{f_0=v_0,\dots,f_{d-1}=v_{d-1}}$ defined by value tuples $v=(v_0,\dots,v_{d-1})\in\set{0,\alpha}^d$ thus form a partition of $X$ consisting of clopen sets. 
However, for every such tuple $v$ not having precisely one entry $\alpha$ we know that the corresponding set $P_v$ is again contained in the complement of $\bigcup_{j=0}^{d-1} D_j'$, so that $P_v=\emptyset$ by the same logic as above. 
Altogether, this shows that the non-empty sets in the so-constructed partition are precisely the $D_i$, proving (i). 

To show that $P(x,D_{i+1})=1$ for each $x\in D_i$, note first that by definition of a $d$-cycle we know $P(x,D_{i+1})\ge P(x,D_{i+1}')=1$ whenever $x\in D_i'$. 
To extend this to $x\in D_i$, we claim that 
\begin{align*}
D_i=\overline{D_i'}.
\end{align*}
Indeed, the inclusion \enquote{$\supset$} follows from the $D_i$ being clopen and the differences $D_i\setminus\overline{D_i'}$ are empty because they are open sets contained in a null set. 
Thus, we may choose a sequence $(x_n)_n$ in $D_i'$ converging to a given $x\in D_i$. 
Writing $x=g\Gamma$, $x_n=g_n\Gamma$, and $\pr\colon G\to X$ for the canonical projection, we find 
\begin{align*}
P(x,D_{i+1})=\mu(\pr^{-1}(D_{i+1})g^{-1})=\lim_{n\to\infty}\mu(\pr^{-1}(D_{i+1})g_n^{-1})=1
\end{align*}
by dominated convergence, since by clopenness the indicator functions of the sets $\pr^{-1}(D_{i+1})g_n^{-1}$ converge pointwise to that of $\pr^{-1}(D_{i+1})g^{-1}$. 

Next, take $g\in\supp(\mu)$ and $x\in D_i$. 
Then for any neighborhood $B$ of $g\in G$ we have $P(x,Bx)>0$. 
If $gx\notin D_{i+1}$, this would contradict $P(x,D_{i+1})=1$ by choosing $B$ small enough. 
Hence, $gD_i\subset D_{i+1}$. 
But the same argument applied to $x\in g^{-1}D_{i+1}$ shows that such an $x$ needs to lie in $D_i$, so that also $g^{-1}D_{i+1}\subset D_i$. 
This proves (ii). 

For (iii), note that $D_i\supset \G_d x$ follows by combining \eqref{convolution_support}, part (ii) above, and clopenness of $D_i$. 
The set $D_i\setminus\G_d x$ is open, since both $D_i$ and $\G_d x$ are clopen (the latter by Lemma~\ref{lem:adapted}, using that also $\mu^{*d}$ is spread out) and 
\begin{align*}
P^n(x,D_i\setminus\G_d x)=0
\end{align*}
for all $n\in\N$: 
Indeed, if $d\mid n$ we have $P^n(x,\G_d x)=1$ and if $d\nmid n$ then $P^n(x,D_j)=1$ for some $j\neq i$. 
The assumed $\psi$-irreducibility thus forces $D_i\setminus\G_d x=\emptyset$, giving (iii). 

It remains to prove (iv). 
Knowing from (ii) that the random walk cycles through the sets $D_0,\dots,D_{d-1}$, $\psi$-irreducibility of the $d$-step random walk on every $D_i$ follows from $\psi$-irreducibility of the whole random walk. 
From \cite[Proposition~5.4.6]{MT} we know that the $d$-step random walk on the full measure subset $D_i'$ of $D_i$ is aperiodic. 
Suppose that the $d$-step random walk in $D_i$ has a period strictly larger than $1$. Then we can apply what we have already proved and deduce that $D_i$ splits into a non-trivial cycle of clopen subsets. By the second statement in (ii), none of the sets in such a cycle can be null sets. 
Restricting to $D_i'$ would thus produce a non-trivial cycle inside $D_i'$, which is a contradiction. 
\end{proof}
It is natural to ask when the particularly desirable aperiodic case $d=1$ occurs. 
\begin{definition}
If $\mu$ has the property that the induced random walk on $\G x$ is $\psi$-irreducible and aperiodic for every $x\in X$, we call $\mu$ \emph{aperiodic on $X$}. 
\end{definition}
\begin{proposition}\label{prop:aperiodicity}
Let $\mu$ be spread out. 
Suppose that 
\begin{itemize}
\item $X$ has finite volume, or that 
\item $\G=\calS$. 
\end{itemize}
Then either one of the following conditions is sufficient for $\mu$ to be aperiodic on $X$: 
\begin{enumerate}
\item $X$ is connected, 
\item $\mu$ is aperiodic in the sense of Definition~\textup{\ref{def:spread_out_aper}}. 
\end{enumerate}
\end{proposition}
\begin{proof}
We first replace the pair $(G,X)$ by $(\G,\G x)$ using Lemma~\ref{lem:adapted}. 
Then $\mu$ is $\Gamma$-adapted (in the finite volume case by Proposition~\ref{prop:fin_vol_adapted}), so that by Proposition~\ref{prop:spread_out_irr} the random walk on $X$ is $\psi$-irreducible with $\psi$ in the Haar measure class. 

Sufficiency of (i) is then evident from Proposition~\ref{prop:periodicity}, since it shows that the sets in a $d$-cycle may be chosen to be clopen. 

For (ii), we argue by contradiction and assume that the period $d$ of the random walk on $X$ is at least $2$. Let us partition $X=\G x$ into clopen sets $D_0,\dots,D_{d-1}$ as in Proposition~\ref{prop:periodicity}. 
Its part (ii) implies that all elements of $\supp(\mu)$ act on the $D_i$ by the cyclic permutation $D_0\mapsto D_1\mapsto\dots\mapsto D_{d-1}\mapsto D_0$. 
Since $\supp(\mu)$ generates $G$ topologically and the $D_i$ are clopen, this yields a continuous homomorphism $\varphi$ from $G=\G$ into the symmetric group of $\set{D_0,\dots,D_{d-1}}$ with image $\varphi(G)\cong \Z/d\Z$. 
But then, the kernel $N=\ker(\varphi)$ is a normal open subgroup of $G$, which contains $[G,G]$ since the quotient $G/N\cong \Z/d\Z$ is abelian, and such that $\supp(\mu)$ is contained in a non-identity--coset of $N$, since $\varphi(\supp(\mu))=\set{1+d\Z}$ under the identification $\varphi(G)\cong \Z/d\Z$. 
Hence, (ii) does not hold. 
\end{proof}
In particular, adapted spread out probability measures are automatically aperiodic on any finite volume quotient when $G$ is connected or a \emph{perfect} group, i.e.\ one with $G=[G,G]$. 
An example of the latter case not covered by the first is $G=\SL_d(\Q_p)$. 
This is an instance of a more general fact. 
\begin{corollary}
Let $k$ be the field $\R$ of real numbers or the field $\Q_p$ of $p$-adic numbers for a prime $p$. Suppose that $G=\mathbb{G}(k)$ for a Zariski connected, simply connected, semisimple algebraic group $\mathbb{G}$ defined over $\Q$ such that $G$ has no compact factors, and let $\mu$ be a spread out probability measure on $G$. 
Then $\mu$ is aperiodic on $X=G/\Gamma$ in both of the following cases: 
\begin{enumerate}
\item $\mu$ is adapted and $X$ has finite volume, 
\item $\mu$ is \emph{strongly adapted}, meaning that $\calS=G$. 
\end{enumerate}
\end{corollary}
\begin{proof}
By \cite[Corollary~2.3.2(b)]{Mar_book} $G$ is perfect, so  Proposition~\ref{prop:aperiodicity}(ii) applies. 
\end{proof}
\subsection{Harris Recurrence}\label{subsec:harris_rec}
As final part of our qualitative analysis, we establish Harris recurrence of spread out random walks for homogeneous spaces $X$ with at most quadratic growth. 
As warm-up, let us show how recurrence can be deduced from what we have already proved in the finite volume case. 
\begin{proposition}\label{prop:harris_rec}
Suppose $\Gamma\subset G$ is a lattice and that the random walk on $X$ induced by $\mu$ is $\psi$-irreducible. 
Then this random walk is positive Harris recurrent. 
\end{proposition}
\begin{proof}
Positive recurrence follows from Proposition~\ref{prop:positive}, since $m_X$ is an invariant probability measure. 
In order to upgrade this to Harris recurrence, we will show that the set $N$ in the decomposition $X=H\sqcup N$ from Theorem~\ref{thm:T_decomposition} must be empty. 
(This theorem can be applied since we know from Proposition~\ref{prop:spread_out_irr} that $\mu$ must be spread out, so that the random walk is a $T$-chain by Proposition~\ref{prop:T_equivalence}.) 
It thus only remains to show that there are no topologically transient points. 
But this is easily seen: 
Proposition~\ref{prop:positive} also implies that $m_X$ is equivalent to $\psi$, so that every non-empty open subset of $X$ is $\psi$-positive. 
Recalling the definition of recurrence from Theorem~\ref{thm:recurrence}, it follows that $U(x,B)=\infty$ for every neighborhood $B$ of any point $x\in X$. 
This precisely means that every point of $X$ is topologically recurrent. 
We thus conclude that $N=\emptyset$, finishing the proof. 
\end{proof}

The remainder of this section is dedicated to the proof of Theorem~\ref{thm:quadratic_recurrence}. 
The following proposition contains the essential lower bound. 
\begin{proposition}\label{prop:lower_bound}
Suppose that $X$ admits a Haar measure $m_X$. 
Let $V$ be a symmetric relatively compact neighborhood of the identity in $G$, $A\subset V\Gamma\subset X$ a positive measure set, and $\mu$ a symmetric probability measure on $G$ with $\supp(\mu)\subset V$. 
Then for $n,\ell\in\N$ satisfying 
\begin{align*}
\ell\ge\sqrt{n\log\frac{16m_X(V^{n+1}\Gamma)}{m_X(A)}}
\end{align*}
we have 
\begin{align*}
\langle P^{2n}\mathds{1}_A,\mathds{1}_A\rangle\ge \frac{m_X(A)^2}{4m_X(V^\ell\Gamma)},
\end{align*}
where $\langle\cdot,\cdot\rangle$ denotes the pairing $\langle \varphi,\psi\rangle=\int_X\varphi\psi\dd m_X$ for measurable functions $\varphi,\psi$ on $X$. 
\end{proposition}
The proof is adapted from Lust-Piquard~\cite{LP_lower_bounds} and contains ideas going back to Carne~\cite{Ca}. 
\begin{proof}
From the defining property $P\varphi(x)=\int_G\varphi(gx)\dd\mu(g)$ of the action of $P$ on measurable functions and invariance of $m_X$ we get that $P$ is a well-defined operator from $L^1(m_X)$ to itself as well as from $L^\infty(m_X)$ to itself, with operator norm bounded by $1$ in both cases. 
By interpolation, the same is true for all $L^p$-spaces. 
Symmetry of $\mu$ implies that $P$ is self-adjoint in the sense that $\langle P\varphi,\psi\rangle=\langle \varphi,P\psi\rangle$ whenever these pairings are defined. 
Using the Cauchy--Schwarz inequality we thus find
\begin{align*}
\langle P^{2n}\mathds{1}_A,\mathds{1}_A\rangle^{1/2}=\norm{P^n\mathds{1}_A}_{L^2(m_X)}\ge \frac{\langle P^n\mathds{1}_A,\mathds{1}_{V^\ell\Gamma}\rangle}{m_X(V^\ell\Gamma)^{1/2}}.
\end{align*}
Writing 
\begin{align*}
\langle P^n\mathds{1}_A,\mathds{1}_{V^\ell\Gamma}\rangle&=\langle P^n\mathds{1}_A,\mathds{1}\rangle-\langle P^n\mathds{1}_A,\mathds{1}_{(V^\ell\Gamma)^c}\rangle=\langle \mathds{1}_A,\underbrace{P^n\mathds{1}}_{=\mathds{1}}\rangle-\langle P^n\mathds{1}_A,\mathds{1}_{(V^\ell\Gamma)^c}\rangle\\
&=m_X(A)-\langle P^n\mathds{1}_A,\mathds{1}_{(V^\ell\Gamma)^c}\rangle,
\end{align*}
we see that it remains to show 
\begin{align}\label{remaining_ineq}
\langle P^n\mathds{1}_A,\mathds{1}_{(V^\ell\Gamma)^c}\rangle\le \frac{m_X(A)}{2}.
\end{align}
We remark that all pairings above are defined since $P^n\mathds{1}_A$ has compact support. 
In fact, by positivity of $P$ we know $0\le P^n\mathds{1}_A\le P^n\mathds{1}_{V\Gamma}$, and the latter function has support in $V^{n+1}\Gamma$ since $\supp(\mu)\subset V$. 

To prove \eqref{remaining_ineq}, we shall use an argument due to Carne~\cite{Ca} (see also \cite[Lemma~1]{LP_lower_bounds}): 
The operator $P^n$ can be written as 
\begin{align*}
P^n=\sum_{0\le k\le n}\alpha_{k,n}Q_k(P),
\end{align*}
where $\alpha_{k,n}=0$ if $n-k$ is odd and otherwise $\alpha_{k,n}=2^{-n+1}\binom{n}{(k+n)/2}$ for $k>0$ and $\alpha_{0,n}=2^{-n}\binom{n}{n/2}$, and $Q_k$ is the $k$-th Chebychev polynomial. 
As the operator $P$ considered on $L^2(m_X)$ is self-adjoint with spectrum contained in $[-1,1]$, the same is true for the operators $Q_k(P)$, since the Chebychev polynomials are real-valued and bounded by $1$ on $[-1,1]$. 
Moreover, $Q_k$ is of degree $k$ so that $Q_k(P)\mathds{1}_A$ is supported in $V^{k+1}\Gamma$ (using the corresponding property of $P^k\mathds{1}_A$ established above). 
Combining these facts we find 
\begin{align*}
\langle P^n\mathds{1}_A,\mathds{1}_{(V^\ell\Gamma)^c}\rangle
&=\sum_{\ell\le k\le n}\alpha_{k,n}\int_{(V^\ell\Gamma)^c}Q_k(P)\mathds{1}_A\dd x\\
&\le \sum_{\ell\le k\le n}\alpha_{k,n}\langle \abs{Q_k(P)\mathds{1}_A},\mathds{1}_{V^{n+1}\Gamma}\rangle\\
&\le m_X(A)^{1/2}m_X(V^{n+1}\Gamma)^{1/2}\sum_{\ell\le k\le n}\alpha_{k,n}\\
&\le 2m_X(A)^{1/2}m_X(V^{n+1}\Gamma)^{1/2}\euler^{-\ell^2/(2n)},
\end{align*}
where the last inequality uses a well-known escape estimate for the symmetric random walk on $\Z$ starting at $0$ (cf.\ e.g.~\cite{Ca}). 
Plugging in the inequality for $\ell$ from the statement of the proposition, we obtain precisely \eqref{remaining_ineq} and the proof is complete. 
\end{proof}
\begin{proof}[Proof of Theorem~\textup{\ref{thm:quadratic_recurrence}}]
Let $V$ be a symmetric relatively compact neighborhood of the identity in $G$ containing $\supp(\mu)$. 
If the random walk on $X$ given by $\mu$ is topologically transient, then by \cite[Theorem~1]{HR} the potential $\sum_{n=1}^\infty P^n\mathds{1}_{V\Gamma}$ is uniformly bounded on $X$. 
In particular, 
\begin{align}\label{contr}
\sum_{n=1}^\infty \langle P^n\mathds{1}_{V\Gamma},\mathds{1}_{V\Gamma}\rangle<\infty.
\end{align}
However, since $X$ has at most quadratic growth, we can apply Proposition~\ref{prop:lower_bound} with $A=V\Gamma$ and $\ell=O(\sqrt{n\log n})$ and find for $n$ large enough that 
\begin{align*}
\langle P^{2n}\mathds{1}_{V\Gamma},\mathds{1}_{V\Gamma}\rangle\ge \frac{m_X(V\Gamma)^2}{4m_X(V^\ell\Gamma)}\ge \frac{C}{n\log n},
\end{align*}
where $C>0$ is a fixed constant. 
This contradicts \eqref{contr}, as $\sum_n\frac{1}{n\log n}=\infty$. 
\end{proof}
In~\S\ref{sec:intro}, Theorem~\ref{thm:quadratic_recurrence} was stated for topological Harris recurrence, since the concept of Harris recurrence was only introduced in~\S\ref{sec:markov_theory}. 
Using the following fact contained in \cite[Theorem~1]{HR}, one also obtains Harris recurrence. 
\begin{proposition}[\cite{HR}]\label{prop:quadr_harris_rec}
Let $\mu$ be an adapted spread out probability measure on~$G$. 
Suppose that there exists a quasi-invariant measure $m_{X,\chi}$ on $X=G/\Gamma$ that is $P$-subinvariant and that the random walk on $X$ induced by $\mu$ is topologically Harris recurrent. 
Then this random walk is $\psi$-irreducible with $\psi\sim[m_X]$ and Harris recurrent. 
\end{proposition}
We point out that, in view of \eqref{quasi_stat}, the measure $m_{X,\chi}$ is $P$-subinvariant if and only if the character $\chi$ satisfies $\int_G\chi\dd\mu\le 1$. 
Therefore, the first condition in the above proposition is satisfied in particular when $X$ admits a Haar measure $m_X$. 
\section{Consequences}\label{sec:consequences}
In this final section we reap the rewards of the preceding work. 
In the finite volume setting, we will establish total variation norm convergence of the laws $\mathcal{L}(\Phi_n)$ in~\S\ref{subsec:law_conv}, see how existence of Lyapunov functions makes this convergence exponentially fast in~\S\ref{subsec:eff_mix} and~\S\ref{subsec:rate}, and present versions of some classical limit theorems in~\S\ref{subsec:limit_thms}. 
We end the article with the proof of the Ratio Limit Theorem~\ref{thm:ratio_limit_thm} in~\S\ref{subsec:ratio_limit}. 

The standing assumptions from the beginning of~\S\ref{sec:spread_out} are still considered to be in effect. 
Let us quickly review the definition of the total variation norm: 
Given a finite signed measure $\nu$ on $X$ it is defined by 
\begin{align*}
\norm{\nu}=\sup_{\abs{f}\le 1}\abs*{\int_Xf\dd\nu},
\end{align*} 
where the supremum is over all measurable functions $f\colon X\to \C$ bounded by $1$. 
With this definition we have 
\begin{align}\label{other_TV}
\sup_{\substack{A\subset X\\\text{measurable}}}\abs{\nu_1(A)-\nu_2(A)}=\tfrac12\norm{\nu_1-\nu_2}
\end{align}
for two probability measures $\nu_1,\nu_2$ on $X$. 
We remark that some authors use the left-hand side above as definition for the total variation distance. 
Due to the factor of $2$ in \eqref{other_TV}, some care needs to be taken when consulting the literature when concerned with the precise value of constants. 
Given a measurable function $V\colon X\to [1,\infty)$, we also define the \emph{$V$-norm} of a finite signed measure $\nu$ as 
\begin{align*}
\norm{\nu}_V=\sup_{\abs{f}\le V}\abs*{\int_Xf\dd\nu}.
\end{align*}
Note that $\norm{\cdot}=\norm{\cdot}_1\le \norm{\cdot}_V$. 
\subsection{Convergence of the Laws}\label{subsec:law_conv}
Using the results from~\S\ref{sec:spread_out}, we can now easily prove convergence to equilibrium of the $n$-step distributions $\mathcal{L}(\Phi_n)$, which is sometimes referred to as \emph{mixing} of the random walk. 

Recall from the discussion in \S\ref{subsec:trans} that spread out random walks on finite volume spaces are automatically $\psi$-irreducible on each orbit $\G x$ with $\psi$ equivalent to $m_{\G x}$, so that the concept of periodicity treated in Theorem~\ref{thm:markov_periodicity} and Proposition~\ref{prop:periodicity} is available. 
\begin{theorem}\label{thm:mixing}
Suppose that $\Gamma\subset G$ is a lattice. 
Let $\mu$ be spread out and $d\in\N$ be the period of the induced random walk on $\G x$ for some $x\in X$. 
Then for any starting distribution $\nu$ on $\G x$ we have 
\begin{align*}
\norm*{\frac1d\sum_{j=0}^{d-1}\mu^{*(n+j)}*\nu-m_{\G x}}\longrightarrow 0
\end{align*}
as $n\to\infty$. 
\end{theorem}
\begin{proof}
By Lemma~\ref{lem:adapted} we may assume without loss of generality that $X=\G x$ and that $\mu$ is adapted. 
Then Corollary~\ref{cor:fin_vol_irred} and Proposition~\ref{prop:harris_rec} together imply that the random walk on $X$ is positive Harris recurrent. 
Its unique invariant probability is $m_X$. 
In the aperiodic case $d=1$ the result thus is a direct consequence of \cite[Theorem~13.3.3]{MT}. 

We will now reduce the general case to the aperiodic one. 
Let $D_0,\dots,D_{d-1}\subset X$ be a $d$-cycle with the properties from Proposition~\ref{prop:periodicity}. 
Writing $\nu$ as a convex combination and using the triangle inequality, we may assume that $\nu$ is supported on one of the $D_i$. 
It will be enough to prove the result for $n$ tending to $\infty$ inside each one of the arithmetic progressions $r+d\N$ for $r=0,\dots,d-1$. 
So let us fix one such $r$ and replace $n$ by $nd+r$ in the claimed statement. 
After renumbering the $D_i$ we may assume that $\mu^{*r}*\nu$ is supported inside $D_0$. 
Setting $\nu_i=\mu^{*(i+r)}*\nu$ for $i=0,\dots,d-1$ we have that $\nu_i$ is supported inside $D_i$ and are left to show that
\begin{align*}
\norm*{\frac1d\sum_{i=0}^{d-1}\mu^{*nd}*\nu_i-m_X}\longrightarrow 0
\end{align*}
as $n\to\infty$. 
However, in view of part (iv) of Proposition~\ref{prop:periodicity}, this follows from the aperiodic case, after writing $m_X=\frac1d(m_{D_0}+\dots+m_{D_{d-1}})$ and applying the triangle inequality once more. 
\end{proof}
\subsection{Lyapunov Functions and Effective Mixing}\label{subsec:eff_mix}
Functions enjoying certain contraction properties under a transition kernel are known as \emph{\textup{(}Foster--\textup{)}Lyapunov functions} and have played a major role in questions of recurrence of dynamical systems since their introduction. 
In our setup, they will produce an exponential rate for the conclusion of Theorem~\ref{thm:mixing}. 

Recall that a function $f\colon X\to[0,\infty]$ is called \emph{proper} if for every $R\in[0,\infty)$ the preimage $f^{-1}([0,R])$ is relatively compact. 
\begin{definition}
A proper Borel function $V\colon X\to[0,\infty]$ is called a \emph{Lyapunov function} for a Markov chain on $X$ given by a transition kernel $P$ if there exist constants $\alpha<1$, $\beta\ge 0$ such that $PV\le \alpha V+\beta$. 
\end{definition}
Such a function should be thought of as directing the dynamics of the Markov chain towards the \enquote{center} of the space, where the function value of $V$ is below some threshold. 
\begin{remark}\label{rmk:lyapunov}
Let us collect some immediate observations about Lyapunov functions. 
\begin{enumerate}
\item If $V$ is a Lyapunov function, then so are $cV$ and $V+c$ for any constant $c>0$. 
In particular, one may impose an arbitrary lower bound on $V$. 
This will be relevant at some points, where we want $V$ to take values $\ge 1$. 
\item Given a function $V'\colon X\to[0,\infty]$ as in the definition of a Lyapunov function, except that $V'$ is contracted by some power $P^{n_0}$ instead of $P$, one can construct a Lyapunov function $V$ by setting 
\begin{align*}
V=\sum_{k=0}^{n_0-1}\alpha^{\frac{n_0-1-k}{n_0}}P^kV'.
\end{align*}
\item By enlarging $\alpha$ and using properness, the contraction inequality in the definition of a Lyapunov function $V$ may be replaced by 
\begin{align*}
PV\le \alpha V+\beta\mathds{1}_K
\end{align*}
for some compact $K\subset X$ (cf.\ \cite[Lemma~15.2.8]{MT}). \qedhere
\end{enumerate}
\end{remark}

The constant function $V\equiv \infty$ always is a Lyapunov function, though one of little use. 
Of greater interest is the existence of Lyapunov functions that are finite on prescribed parts of the space, or even finite everywhere. 
\begin{definition}\label{def:CH}
We say that a subset $A\subset X$ is \emph{Lyapunov small} for a random walk on $X$ given by $\mu$ if the random walk admits a Lyapunov function $V_A\colon X\to[0,\infty]$ that is bounded on $A$. 
We say the random walk satisfies the \emph{contraction hypothesis} if every compact subset $K\subset X$ is Lyapunov small. 
\end{definition}

Constructions of Lyapunov functions on quotients of semisimple Lie groups were given by Eskin--Margulis~\cite{EM} and Benoist--Quint~\cite{BQ_rec}. 
We record the consequences for spread out random walks in the example below. 
Recall that a measure $\mu$ on a Lie group $G$ with Lie algebra $\g$ is said to have \emph{finite exponential moments in the adjoint representation} if for sufficiently small $\delta>0$ 
\begin{align*}
\int_G\norm{\Ad(g)}_{\mathrm{op}}^\delta\dd\mu(g)<\infty,
\end{align*}
where $\norm{\cdot}_{\mathrm{op}}$ denotes an operator norm on $\Aut(\g)$. 
\begin{example}\label{ex:lyapunov_existence}\leavevmode
\begin{enumerate}
\item (\cite{BQ_rec}) Let $G$ be a real Lie group and $\mu$ an adapted spread out probability measure on $G$ with finite exponential moments in the adjoint representation. 
Suppose that the Zariski closure of $\Ad(G)$ in $\Aut(\g)$ is Zariski connected and semisimple. 
Then the random walk on $X=G/\Gamma$ given by $\mu$ satisfies the contraction hypothesis. 
Using the setup in \cite[Section~7]{BQ_rec}, a similar statement can also be made about $p$-adic Lie groups. 
\item (\cite{EM}) Let $G=\mathbb{G}(\R)$ be the group of real points of a Zariski connected semisimple algebraic group $\mathbb{G}$ defined over $\R$ such that $G$ has no compact factors and $\mu$ a spread out probability measure on $G$ with finite exponential moments in the adjoint representation. 
Then the random walk on $X=G/\Gamma$ admits a continuous and everywhere finite Lyapunov function. \qedhere
\end{enumerate}
\end{example}
%
Equipped with these concepts, we can now explain how Lyapunov functions make mixing of spread out random walks effective. 
For the sake of simplicity, we only state the result in the adapted and aperiodic case. 
We have repeatedly seen that the former is no restriction, and the corresponding statements in the periodic case can be obtained by employing similar reductions as in the proof of Theorem~\ref{thm:mixing}. 
\begin{theorem}\label{thm:eff_mixing}
Suppose that $\Gamma\subset G$ is a lattice. 
Let $\mu$ be an adapted spread out probability measure on $G$ such that the random walk on $X$ given by $\mu$ is aperiodic. 
\begin{enumerate}
\item For every Lyapunov small subset $A\subset X$ there is a constant $\kappa=\kappa(A)>0$ such that for every $n\in\N$ 
\begin{align*}
\sup_{x\in A}\norm*{\mu^{*n}*\delta_x-m_X}\ll_A \mathrm{e}^{-\kappa n}.
\end{align*}
In particular, this holds for all compact subsets $A=K$ of $X$ if the random walk satisfies the contraction hypothesis. 
\item If the random walk admits an everywhere finite Lyapunov function $V\ge 1$, then there is a constant $\kappa>0$ such that for all $n\in\N$ 
\begin{align*}
\sup_{x\in X}\frac1{V(x)}\norm*{\mu^{*n}*\delta_x-m_X}_V\ll \mathrm{e}^{-\kappa n}.
\end{align*}
\end{enumerate}
\end{theorem}
\begin{proof}
As in the proof of Theorem~\ref{thm:mixing} we see that the random walk on $X$ is positive Harris recurrent with unique invariant probability $m_X$. 
To establish (i), recall that Remark~\ref{rmk:lyapunov} allows us to assume that $V_A$ is bounded below by $1$ and that $PV_A\le \alpha V_A+\beta\mathds{1}_L$ for some compact $L\subset X$. 
Since compact sets are petite for $T$-chains (\cite[Theorem~6.2.5]{MT}, see \cite[p.~117]{MT} for the definition of petite sets), $V_A$ satisfies the condition in (iii) of \cite[Theorem~15.0.1]{MT}. 
Since $V_A$ is bounded on $A$, the claim follows from the last statement of that theorem. 
With the same arguments, (ii) follows from \cite[Theorem~16.1.2]{MT}. 
\end{proof}
On compact spaces, one may always choose $V=1$ as Lyapunov function. 
This immediately gives the following corollary. 
\begin{corollary}\label{cor:compact}
Suppose in addition to the assumptions in Theorem~\textup{\ref{thm:eff_mixing}} that $X$ is compact. 
Then there exists $\kappa>0$ such that for all $n\in\N$ 
\begin{equation*}
\pushQED{\qed}
\sup_{x\in X}\norm{\mu^{*n}*\delta_x-m_X}\ll \mathrm{e}^{-\kappa n}.\qedhere
\popQED
\end{equation*}
\end{corollary}
\begin{proof}[Proof of Theorem~\textup{\ref{thm:main_thm}}]
That $\G x_0$ is clopen was part of Lemma~\ref{lem:adapted}. 
In view of Proposition~\ref{prop:aperiodicity}, \eqref{equidist} follows from Theorem~\ref{thm:mixing}. 

To obtain the statement about effective mixing, first ensure that the Lyapunov function is bounded below by $1$ using Remark~\ref{rmk:lyapunov}(i) and then apply Theorem~\ref{thm:eff_mixing}(ii) to each of the finitely many $\G$-orbits intersecting the compact set $K$. 
The conclusion follows, since $\norm{\cdot}\le \norm{\cdot}_V$ and $V$ is bounded on $K$ by the assumed continuity. 

The final remark about existence of Lyapunov functions is Example~\ref{ex:lyapunov_existence}(ii). 
\end{proof}
\begin{proof}[Proof of Corollary~\textup{\ref{cor:conn}}]
In both cases of the corollary
\begin{itemize}
\item $\mu$ is aperiodic on $X$ by Proposition~\ref{prop:aperiodicity}, and
\item we have $\G x_0=X$ for all $x_0\in X$, using Corollary~\ref{cor:full_orbit} or adaptedness of $\mu$, respectively. 
\end{itemize}
The statement now follows from Theorem~\ref{thm:mixing}. 
\end{proof}
\subsection{Small Sets and Mixing Rates}\label{subsec:rate}
Given the existence of Lyapunov functions or compactness of the state space, we know from~\S\ref{subsec:eff_mix} that the convergence 
\begin{align*}
\mu^{*n}*\delta_x\overset{n\to\infty}{\longrightarrow} m_X
\end{align*}
happens with exponential speed. 
As long as the value of the exponent and the implicitly appearing constants are unknown, this does not yet give any information about the actual variation distance between $\mu^{*n}*\delta_x$ and $m_X$ for any given $n\in\N$. 
In this subsection, we will address this issue. 
The crucial concept is the following. 
\begin{definition}
Let $P$ be a transition kernel on $X$. 
A set $A\subset X$ is called \emph{$(n,\epsilon)$-small} for an integer $n\in\N$ and $\epsilon>0$ if there exists a probability measure $\lambda$ on $X$ such that 
\begin{align*}
P^n(x,\cdot)\ge \epsilon\lambda
\end{align*}
for all $x\in A$. 
If $A$ is $(n,\epsilon)$-small for some $n\in\N$ and $\epsilon>0$, $A$ is called \emph{small}. 
\end{definition}
Small sets are in fact one of the central notions on which the whole theory of general state space Markov chains is built. 
Their significance lies in the fact that they provide the Markov chain with a regenerative structure: 
After each return to $A$, there is a positive probability of taking the next step according to the fixed measure $\lambda$. 
This structure also plays an important role when trying to establish bounds on the speed of convergence. 
The simplest result in this direction assumes that the whole state space is small, which is known as the \emph{Doeblin condition}. 
\begin{theorem}[{\cite[Theorem~16.2.4]{MT}}]\label{thm:doeblin}
Suppose the whole state space $X$ is $(n_0,\epsilon)$-small for a Markov chain with transition kernel $P$ and invariant probability $\pi$. 
Then for all $n\in\N$ and any starting distribution $\nu$ on $X$ we have 
\begin{align*}
\norm{\nu P^n-\pi}\le 2(1-\epsilon)^{\lfloor n/n_0\rfloor}.
\end{align*}
\end{theorem}
When the state space is not small, a rate of convergence as simple as above may not be available. 
We shall use the following result due to Rosenthal~\cite{R}. 
\begin{theorem}[{\cite[Theorem~5]{R}}]\label{thm:rate}
Given a transition kernel $P$ on $X$, denote the product kernel by $\overline{P}((x,y),\cdot)=P(x,\cdot)\otimes P(y,\cdot)$. 
Let $\pi$ be a $P$-invariant probability measure on $X$. 
Suppose there exists an $(n_0,\epsilon)$-small set $A$ and a measurable function $h\ge 1$ on $X\times X$ together with a constant $\overline{\alpha}<1$ such that 
\begin{align*}
\overline{P}h(x,y)\le \overline{\alpha} h(x,y)
\end{align*}
for all $(x,y)\notin A\times A$. 
Then, with $R=\sup_{(x,y)\in A\times A}\overline{P}^{n_0}h(x,y)$, we have for all $j,n\in\N$ and any starting distribution $\nu$ on $X$ 
\begin{align*}
\norm{\nu P^n-\pi}\le 2(1-\epsilon)^{\lfloor j/n_0\rfloor}+2\overline{\alpha}^{n-jn_0+1}R^{j-1}\int_{X\times X}h\dd(\nu\otimes\pi).
\end{align*} 
\end{theorem}

In order to apply these theorems, we see that it is important to identify small sets for spread out random walks. 
From a qualitative point of view, this task is not too difficult. 
\begin{proposition}\label{prop:compact_small}
Let $\mu$ be a probability measure on $G$. 
Suppose that the random walk on $X=G/\Gamma$ is $\psi$-irreducible and aperiodic. 
Then every compact subset $K\subset X$ is small. 
\end{proposition}
\begin{proof}
Combining Propositions~\ref{prop:spread_out_irr} and~\ref{prop:T_equivalence} we see that the random walk on $X$ is a $T$-chain. 
Compact sets are thus petite by \cite[Theorem~6.2.5]{MT}. 
By aperiodicity and \cite[Theorem~5.5.7]{MT}, they are also small. 
\end{proof}
In the case of a compact state space, we therefore immediately get the following. 
\begin{theorem}\label{thm:compact_rate}
Let $\mu$ be adapted and spread out. 
Suppose that $X$ is compact and that the random walk on $X$ given by $\mu$ is aperiodic. 
Then $X$ is $(n_0,\epsilon)$-small for some $n_0\in\N$ and $\epsilon>0$ and for any starting distribution $\nu$ on $X$ we have 
\begin{align*}
\norm{\mu^{*n}*\nu-m_X}\le 2(1-\epsilon)^{\lfloor n/n_0\rfloor}
\end{align*}
for every $n\in\N$. 
\end{theorem}
\begin{proof}
Note that the random walk is $\psi$-irreducible by Corollary~\ref{cor:fin_vol_irred}. 
Then $X$ is small by Proposition~\ref{prop:compact_small} and Theorem~\ref{thm:doeblin} gives the result. 
\end{proof}
Unfortunately, so far we still have no information about the value of $n_0$ and $\epsilon$. 
We shall now outline a hands-on approach to find them. 
The idea is the following: 
Denote by $f_n\colon G\to[0,\infty)$ the density of the part of $\mu^{*n}$ absolutely continuous with respect to a right Haar measure on $G$ and endow $X$ with the quasi-invariant measure $m_{X,\Delta}$ coming from the modular character $\Delta$ of $G$. 
Then the probability of going from $x\in X$ to $y\in X$ in $n$ steps (using only the continuous part) is represented by the quantity 
\begin{align}\label{rhs_density}
\sum_{\gamma\in\Gamma}f_n(y\gamma x^{-1}),
\end{align} 
which is a function on $X\times X$. 
Note that for fixed $x$, the sum is finite for a.e.\ $y$, since $f_n$ is integrable. 
Hence, the minorization condition in the definition of small sets is certainly satisfied on $A$ for the right-hand side 
\begin{align}\label{rhs_measure}
\inf_{x\in A}\sum_{\gamma\in\Gamma}f_n(y\gamma x^{-1})\dd m_{X,\Delta}(y),
\end{align}
which can be thought of as the lower envelope of the shifts by elements of $A$ of the density \eqref{rhs_density}. 
The remaining question is whether this measure is non-trivial. 
Intuitively, the spread out assumption should guarantee this for large $n$, at least when the shifting set $A$ is not too large, say compact. 
That this is indeed true is the content of the next lemma. 
\begin{lemma}\label{lem:small}
Let $\mu$ be spread out and $\Gamma$-adapted. 
Suppose that the induced random walk on $X$ is aperiodic. 
Then for every compact subset $K\subset X$ there exists an integer $n_0\in\N$ such that the measure \eqref{rhs_measure} has positive mass $\epsilon>0$. 
In particular, $K$ is $(n_0,\epsilon)$-small. 
\end{lemma}
\begin{proof}
By Proposition~\ref{prop:spread_out_irr} the random walk is $\psi$-irreducible with $\psi\sim[m_X]$. 
Enlarging $K$ if necessary, we may assume that $K$ is $[m_X]$-positive. 
Proposition~\ref{prop:compact_small} implies that $K$ is small. 
Choose $n_1,\epsilon_1,\lambda$ as in the definition of a small set. 
From \cite[Proposition~5.5.4(ii)]{MT} it then follows that $\lambda\ll [m_X]$. 
Hence, there exists an $[m_X]$-positive set $A$ such that $\lambda|_A$ belongs to the Haar measure class restricted to $A$. 
Let us now split the transition kernels $P^n$ into the absolutely continuous and singular parts with respect to $[m_X]$. 
As explained before the statement of the lemma, the absolutely continuous part can then be written as $T_n(x,\dd y)=p_n(x,y)\dd m_{X,\Delta}(y)$ with $p_n\colon X\times X\to [0,\infty]$ given by \eqref{rhs_density}. 
According to \cite[Proposition~1.2]{O} (cf.\ also \cite[Theorem~5.2.1]{MT} and its proof) 
\begin{itemize}
\item the densities $p_n$ can be modified on an $[m_X]$--null~set (in the $y$-coordinate) so that they satisfy 
\begin{align}\label{better_chap_kol}
p_{m+n}(x,z)\ge \int_XP^m(x,\dd y)p_n(y,z)
\end{align}
for all $x,z\in X$ and $m,n\in\N$, and
\item there exists an $[m_X]$-positive set $C\subset A$ and $n_2\in\N$ such that 
\begin{align}\label{cont_bound}
p_{n_2}(x,y)\ge \delta
\end{align}
for all $x,y\in C$ and some fixed $\delta>0$. 
\end{itemize}
By construction of $A$ we then know $\lambda(C)>0$, so that for all $x\in K$ 
\begin{align}\label{general_bound}
P^{n_1}(x,C)\ge \epsilon_1\lambda(C)>0.
\end{align}
Combining \eqref{better_chap_kol}, \eqref{cont_bound} and \eqref{general_bound}, we find for $x\in K$, $z\in C$ and $n_0=n_1+n_2$ that 
\begin{align*}
p_{n_0}(x,z)\ge \int_CP^{n_1}(x,\dd y)p_{n_2}(y,z)\ge \delta P^{n_1}(x,C)\ge \delta\epsilon_1\lambda(C).
\end{align*}
Hence, the mass of \eqref{rhs_measure} is at least 
\begin{align*}
\int_C\inf_{x\in A}p_{n_0}(x,z)\dd m_{X,\Delta}(z)\ge \delta\epsilon_1\lambda(C)m_{X,\Delta}(C)>0,
\end{align*}
which is the claim. 
\end{proof}
\begin{example}
Let us illustrate the method above by calculating a rate of convergence in a concrete instance of Example~\ref{ex:torus}. 
We set $n=2$, $a=(\begin{smallmatrix}2&1\\1&1\end{smallmatrix}),\,b=(\begin{smallmatrix}1&1\\1&2\end{smallmatrix})$, $\mu_{\mathrm{lin}}=\frac12(\delta_a+\delta_b)$, $v=e_1=(1,0)^t$, and assume that $\lambda$ has a component with a density $f$ bounded below by $\delta>0$ on $[0,1]$. 
Let us see how we need to choose $n_0$. 
We certainly cannot use $n_0=1$, since $\mu=\mu_{\mathrm{lin}}\otimes \lambda_{e_1}$ is singular with respect to Haar measure. 
If we denote the first two displacements by the random variables $D_1,D_2$, the possible two-step transformations are 
\begin{align*}
a^2\colon&x\mapsto a(ax+D_1v)+D_2v=a^2x+D_1ae_1+D_2e_1=a^2x+\begin{pmatrix}2D_1+D_2\\D_1\end{pmatrix},\\
ab\colon&x\mapsto a(bx+D_1v)+D_2v=abx+D_1ae_1+D_2e_1=abx+\begin{pmatrix}2D_1+D_2\\D_1\end{pmatrix},\\
ba\colon&x\mapsto b(ax+D_1v)+D_2v=bax+D_1be_1+D_2e_1=bax+\begin{pmatrix}D_1+D_2\\D_1\end{pmatrix},\\
b^2\colon&x\mapsto b(bx+D_1v)+D_2v=b^2x+D_1be_1+D_2e_1=b^2x+\begin{pmatrix}D_1+D_2\\D_1\end{pmatrix}.
\end{align*}
Since $D_1,D_2$ are i.i.d.\ with density $f$, the densities of the above displacements are 
\begin{align*}
g_{a^2}(s,t)=g_{ab}(s,t)=f(t)f(t-2s),\;g_{ba}(s,t)=g_{b^2}(s,t)=f(t)f(t-s)
\end{align*}
for $s,t\in\R^2$, which, by our assumption, are all bounded below by $\delta^2$ on a fundamental domain for $\mathbb{T}^2$. 
Hence, the mass of the measure \eqref{rhs_measure} for $n_0=2$ and $A=\mathbb{T}^2$ is at least $\delta^2$, so that Theorem~\ref{thm:compact_rate} produces the bound 
\begin{align*}
\norm{\mu^{*n}*\nu-m_X}\le 2(1-\delta^2)^{\lfloor n/2\rfloor}
\end{align*}
for all $n\in\N$, where $\nu$ is an arbitrary starting distribution. 
(Note that aperiodicity is guaranteed here in view of Proposition~\ref{prop:aperiodicity}.) 
\end{example}

We now turn our attention to the case of a non-compact finite-volume space $X$. 
Here we shall assume that the random walk on $X$ admits a Lyapunov function $V$ and apply Theorem~\ref{thm:rate} in a similar way as in the proof of \cite[Theorem~12]{R}. 

The set $A$ from Theorem~\ref{thm:rate} is going to be the sublevel set 
\begin{align*}
A=\set{x\in X\for V(x)\le d}
\end{align*}
for some $d>1$, and $h$ is going to be defined as
\begin{align*}
h(x,y)=1+V(x)+V(y)
\end{align*}
for $x,y\in X$. 
Note that $A$ is relatively compact since $V$ is proper and thus a small set by Proposition~\ref{prop:compact_small}. 
Let now $\alpha,\beta$ be the constants associated to the Lyapunov function $V$ and $(x,y)\notin A\times A$. 
Then $h(x,y)>1+d$, and thus we find 
\begin{align*}
\overline{P}h(x,y)&=1+PV(x)+PV(y)\\
&\le 1-\alpha+\alpha h(x,y)+2\beta\\
&\le \br*{\frac{1-\alpha+2\beta}{1+d}+\alpha}h(x,y)\\
&=\frac{1+\alpha d+2\beta}{1+d}h(x,y).
\end{align*}
Choosing $\overline{\alpha}=\frac{1+\alpha d+2\beta}{1+d}$, this will be the contraction condition in Theorem~\ref{thm:rate}. 
In order for $\overline{\alpha}$ to be less than $1$, $d$ needs to be chosen so that 
\begin{align*}
d>\frac{2\beta}{1-\alpha}.
\end{align*}
This choice of $d$ determines the set $A$. 

By iterating the Lyapunov property of $V$ and using the definition of $A$, the value of $R$ in Theorem~\ref{thm:rate} can be estimated as 
\begin{align*}
R&=1+2\sup_{x\in A}P^{n_0}V(x)\\
&\le 1+2\br*{\alpha^{n_0}\sup_{x\in A}V(x)+\beta\frac{1-\alpha^{n_0}}{1-\alpha}}\\
&\le 1+2\br*{\alpha d+\frac{\beta}{1-\alpha}}.
\end{align*}

For the integral of $h$, note first that $V$ is necessarily $m_X$-integrable by the equivalence of (i) and (iii) in \cite[Theorem~14.0.1]{MT} (use $f=(1-\alpha)V$), so that $P$-invariance of $m_X$ and the contraction property of $V$ yield $\int_X V\dd m_X\le \frac{\beta}{1-\alpha}$. 
It follows that 
\begin{align*}
\int_{X\times X}h\dd(\nu\otimes m_X)&\le 1+\int_X V\dd\nu+\int_X V\dd m_X\\
&\le 1+\int_X V\dd\nu+\frac{\beta}{1-\alpha}.
\end{align*}
Putting everything together, we arrive at the following theorem. 
\begin{theorem}
Let $\Gamma\subset G$ be a lattice and $\mu$ be an adapted spread out probability measure on $G$. 
Suppose that the random walk on $X$ given by $\mu$ is aperiodic and admits a Lyapunov function $V$ with $PV\le \alpha V+\beta$ for some $\alpha<1$, $\beta\ge 0$. 
Let $d>\frac{2\beta}{1-\alpha}$ and set $A=\set{x\in X\for V(x)\le d}$. 
Then $A$ is $(n_0,\epsilon)$-small for some $n_0\in\N$ and $\epsilon>0$ and for any starting distribution $\nu$ on $X$ with $\int_X V\dd\nu<\infty$ we have for all $j,n\in\N$ 
\begin{align*}
\norm{\mu^{*n}*\nu-m_X}\le 2(1-\epsilon)^{\lfloor j/n_0\rfloor}+2\overline{\alpha}^{n-jn_0+1}R^{j-1}\br*{1+\int_X V\dd\nu+\frac{\beta}{1-\alpha}},
\end{align*}
where $\overline{\alpha}=\frac{1+\alpha d+2\beta}{1+d}<1$ and $R=1+2\br[\big]{\alpha d+\frac{\beta}{1-\alpha}}$.\qed
\end{theorem}
Note that by introducing the relationship $j=\lfloor n/k\rfloor$ for some $k\in\N$ for which $\overline{\alpha}^{k-n_0}R<1$, the right-hand side above decays exponentially in $n$, and moreover that all the constants are given explicitly in terms of the starting distribution $\nu$, the Lyapunov function $V$ together with its parameters, and the measure $\mu$. 
\subsection{Limit Theorems}\label{subsec:limit_thms}
In the setting of Theorem~\ref{thm:BQ} there is another conclusion that can be drawn, concerning the distribution of typical trajectories: 
For every $x_0\in X$ and $\mu^{\otimes\N}$-a.e.\ $(g_n)_n\in G^\N$ it holds that 
\begin{align*}
\frac1n\sum_{k=0}^{n-1}\delta_{g_k\dotsm g_1x_0}\longrightarrow \nu_{x_0}
\end{align*}
as $n\to\infty$ in the weak* topology (see \cite[Theorem~1.3]{BQ3}). 
In other words, for every $f\in C_c(X)$ we have 
\begin{align*}
\frac1n\sum_{k=0}^{n-1}f(\Phi_k)\longrightarrow \int_Xf\dd\nu_{x_0}\quad\Pro_{x_0}\text{-a.s.}
\end{align*}
as $n\to\infty$, where, as before, $\Phi_k$ is given by \eqref{Phi_def} and stands for the location after $k$-steps of the random walk. 
Until now, we have not yet touched upon the validity of such a \emph{Strong Law of Large Numbers} in the spread out case; an omission that will be corrected now. 

To fix the terminology, let us quickly review three of the classical limit theorems in the context of Markov chains. 
For brevity, we shall use the notation 
\begin{align*}
\Sigma_n(f)=\sum_{k=0}^{n-1}f(\Phi_k)
\end{align*}
for a function $f$ on $X$. 
\begin{definition}
Consider the random walk on $X$ given by a probability measure $\mu$ on $G$. 
Let $f\colon X\to\R$ be a real-valued $m_X$-integrable function on $X$. 
We say 
\begin{itemize}
\item that the \emph{Strong Law of Large Numbers \textup{(}SLLN\textup{)}} holds for $f$ if for every $x\in X$ 
\begin{align*}
\lim_{n\to\infty}\tfrac{1}{n}\Sigma_n(f)=\int_Xf\dd m_X\quad\Pro_x\text{-a.s.,}
\end{align*}
\item that the \emph{Central Limit Theorem \textup{(}CLT\textup{)}} holds for $f$ if there exists a constant $\gamma_f\in[0,\infty)$ such that for the centered function $\overline{f}=f-\int_Xf\dd m_X$ and under each $\Pro_x$ we have convergence in distribution 
\begin{align*}
\tfrac{1}{\sqrt{n}}\Sigma_n(\overline{f})\overset{\mathrm{d}}{\longrightarrow}N(0,\gamma_f^2),
\end{align*}
where $N(0,\gamma_f^2)$ denotes the normal distribution with mean $0$ and variance $\gamma_f^2$ (to be understood as the Dirac distribution at $0$ in the degenerate case $\gamma_f=0$), and 
\item that the \emph{Law of the Iterated Logarithm \textup{(}LIL\textup{)}} holds for $f$ if for this constant $\gamma_f$ and every $x\in X$ 
\begin{align*}
\limsup_{n\to\infty}\tfrac{1}{\sqrt{2n\log\log(n)}}\Sigma_n(\overline{f})=\gamma_f\quad\Pro_x\text{-a.s.}
\end{align*}
\end{itemize}
\end{definition}
The remarkable fact is that spread out random walks always satisfy the SLLN, and satisfy the CLT and LIL as soon as they admit an everywhere finite Lyapunov function. 
\begin{theorem}
Let $\Gamma\subset G$ be a lattice and $\mu$ be spread out and adapted. 
Then: 
\begin{enumerate}
\item The SLLN holds for every $m_X$-integrable function on $X$. 
In particular, for $\mu^{\otimes\N}$-a.e.\ $(g_n)_n\in G^\N$ we have 
\begin{align*}
\frac1n\sum_{k=0}^{n-1}\delta_{g_k\dotsm g_1x_0}\longrightarrow m_X
\end{align*}
as $n\to\infty$ in the weak* topology. 
\item Suppose the random walk admits an everywhere finite Lyapunov function $V\colon X\to[1,\infty)$ and let $f\colon X\to\R$ be measurable and satisfy $f^2\le V$. 
Then for the centered function $\overline{f}=f-\int_Xf\dd m_X$, the asymptotic variance 
\begin{align*}
\gamma_f^2=\lim_{n\to\infty}\tfrac1n\E_{m_X}\sqbr[\big]{\Sigma_n(\overline{f})^2}
\end{align*}
exists and is finite, and the CLT and LIL hold for $f$ and this number $\gamma_f$. 
\end{enumerate}
\end{theorem}
\begin{proof}
Combining Corollary~\ref{cor:fin_vol_irred} and Proposition~\ref{prop:harris_rec} we know that the random walk on $X$ is a positive Harris recurrent Markov chain with invariant probability $m_X$. 
Part (i) thus follows from \cite[Theorem~17.1.7]{MT}, noting for the second claim that $C_c(X)$ is separable. 
Under the assumptions of (ii), Theorem~\ref{thm:eff_mixing}(ii) ensures that the conditions of \cite[Theorem~17.0.1]{MT} are satisfied, and everything follows from that theorem. 
\end{proof}
\subsection{Proof of the Ratio Limit Theorem}\label{subsec:ratio_limit}
It remains to prove the ratio limit theorem for the infinite volume case. 
\begin{proof}[Proof of Theorem~\textup{\ref{thm:ratio_limit_thm}}]
For both statements, it suffices to consider the case in which also $f_1\ge 0$. 
By Proposition~\ref{prop:quadr_harris_rec}, the random walk on $X$ given by $\mu$ is Harris recurrent with invariant measure $m_X$. 
The first statement of the theorem thus immediately follows by combining \cite[Corollary~8.4.3]{revuz} and \cite[Theorem~6.6.5]{revuz}. 

It remains to prove \eqref{strong_ratio} under the additional assumptions that $\mu$ is symmetric and aperiodic. 
In view of Proposition~\ref{prop:aperiodicity}, aperiodicity of $\mu$ implies aperiodicity of the random walk. 
Let $A\subset X$ be a small set with positive and finite $m_X$-measure, say with $P^k(x,\cdot)\ge \epsilon\lambda$ for all $x\in A$. 
In view of \cite[Proposition~5.2.4(iii)]{MT} we may assume that $\lambda(A)>0$, and the discussion in \cite[\S5.4.3]{MT} shows that we may take $k$ to be even. 
With similar arguments as in the proof of Lemma~\ref{lem:small}, after shrinking $A$ and $\epsilon$ we may even assume that $\lambda=m_A=\frac{1}{m_X(A)}m_X|_A$ is the normalized restriction of $m_X$ to $A$ (cf.\ also Orey's $C$-set theorem \cite[Theorem~2.1]{O}). 
Then \cite[Theorem~2.1]{LS_reversible} implies 
\begin{align*}
\lim_{m\to\infty}\frac{\lambda P^{2m}(A)}{\lambda P^{2(m-1)}(A)}=1.
\end{align*}
Applying this $k/2$ times, we see that also 
\begin{align*}
\lim_{m\to\infty}\frac{\lambda P^{km}(A)}{\lambda P^{k(m-1)}(A)}=\lim_{m\to\infty}\frac{\lambda P^{km}(A)}{\lambda P^{km-2}(A)}\frac{\lambda P^{km-2}(A)}{\lambda P^{km-4}(A)}\dotsm\frac{\lambda P^{km-k+2}(A)}{\lambda P^{km-k}(A)}=1.
\end{align*}
This shows that all the assumptions in~\cite{N_ratio_limit} are satisfied. 

In view of \cite[Theorem~1(ii)]{N_ratio_limit}, it remains to argue that compactly supported bounded measurable functions on $X$ and compactly supported probability measures on $X$ with bounded density with respect to $m_X$ are \enquote{small} in Nummelin's sense. 
Identifying such measures with their density and noting that by symmetry of $\mu$ the action of $P$ on functions and measures respects this identification, we see that it suffices to show this claim for functions. 
For this, by \cite[Corollary~2.4]{N_ratio_limit}, we need only show that for every compact subset $K\subset X$ there exists $N\in\N$ such that $\sum_{n=0}^NP^n\mathds{1}_A$ is bounded away from $0$ on $K$. 
However, since the random walk is a $T$-chain and compact sets are petite for $T$-chains (\cite[Theorem~6.2.5(ii)]{MT}), the latter follows from \cite[Proposition~5.5.5(i)]{MT} and \cite[Proposition~5.5.6(i)]{MT}, as $m_X(A)>0$. 
\end{proof}
Invoking \cite[Theorem~1(i)]{N_ratio_limit}, this proof also justifies the claim in Remark~\ref{rmk:conjecture}: 
Taking for $A$ the same small set as above, we see that \eqref{sufficient} implies 
\begin{align*}
\frac{\int_Xf_i\dd(\mu^{*n}*\delta_{x_i})}{(\mu^{*n}*m_A)(A)}\longrightarrow\frac{\int_Xf_i\dd m_X}{m_X(A)}
\end{align*}
as $n\to\infty$ for $i=1,2$. 
Taking the quotient yields \eqref{strong_ratio} with $\delta_{x_i}$ in place of $\nu_i$. 

\bibliographystyle{plain}
\bibliography{refs}

\end{document}